\documentclass{siamltex}
\usepackage{amsmath} 	    % für erweiterte mathematische Konstrukte
\usepackage{amsfonts,amssymb}    	    % für erweiterte mathematische Schriften und Symbole

\usepackage{mathtools}

\usepackage{wrapfig}
\usepackage{todonotes}

\usepackage[top=3cm, bottom=3cm, left=3.5cm, right=3.5cm, heightrounded, marginparwidth=3.1cm, marginparsep=5mm]{geometry}                                          % spacing at boundary
\newcommand{\Qeps}{\cQ_{\eps}}

\def\eps{\varepsilon}

\def\GQ{\Gamma_{\!\cQ}}

\def\ra{\rangle}
\def\la{\langle}

\newcommand{\R}{\mathbb{R}}

\newcommand{\bn}{\mathbf{n}}
\newcommand{\bw}{\mathbf{w}}

\newcommand{\cQ}{\mathcal{Q}}
\newcommand{\cO}{\mathcal{O}}

\newcommand{\sign}{\text{sign}}

\newtheorem{remark}{Remark}[section]
\newtheorem{assumption}{Assumption}
\newtheorem{example}{Example}
\newtheorem{classification}{Classification}
\newcommand{\rev}[1]{{#1}}

\begin{document}

\title{Well-posedness of the heat equation in domains with topological transitions}\thanks{The author M.O. was supported in part by the U.S. National 	Science Foundation under award DMS-2408978.}

\author{Maxim A. Olshanskii\thanks{Department of Mathematics, University of Houston, Houston, Texas 77204-3008 (maolshanskiy@uh.edu).}
	\and Arnold Reusken\thanks{Institut f\"ur Geometrie und Praktische  Mathematik, RWTH-Aachen
		University, D-52056 Aachen, Germany (reusken@igpm.rwth-aachen.de).}
	}
\maketitle

\begin{abstract}
	We analyze a linear parabolic equation with homogeneous Dirichlet boundary conditions posed in domains whose evolution may involve topological transitions. The domains are described as sublevel sets of a smooth space–time level set function, allowing for transitions such as domain splitting and merging and the creation or vanishing of islands and holes.
	We introduce anisotropic space–time function spaces that extend the classical Bochner spaces used in cylindrical domains and establish key functional-analytic properties of these spaces, including the density of compactly supported smooth functions. This framework enables the application of the Babu\v{s}ka–Banach theorem, yielding existence, uniqueness, and \emph{a priori} estimates for weak solutions.
	The analysis applies to domain evolutions generated by level set functions with isolated nondegenerate critical points, which correspond to the generic topology changes classified by Morse theory in two and three spatial dimensions.
	\end{abstract}

\section{Introduction}\label{intro}

Phenomena and processes involving changes in topology are commonly observed in nature and engineering applications. Examples include drop breakage, membrane fusion, bubble coalescence, cell division, and phase transitions. Mathematical models of these and related phenomena may involve partial differential equations posed in domains undergoing topological transitions. In this paper, we address the question of well-posedness for a linear parabolic equation posed in a time-dependent domain that may change its topology in the course of evolution.

Well-posedness of parabolic equations in stationary domains is well understood~\cite{Evans,Wloka1987}. There is also a substantial body of literature devoted to parabolic problems in non-cylindrical space–time domains; see, for example,~\cite{lions1957,acquistapace1987unified,brown1997weak,savare1997parabolic,bonaccorsi2001variational,kuliev2007extension,kloeden2008pullback,Paronetto2013,alphonse2015abstract,kim2018existence}. Topics that are addressed include well-posedness under low regularity assumptions on the space-time domain or the (diffusion) coefficients in the equation (e.g., \cite{bonaccorsi2001variational}), particular non-standard boundary conditions (e.g., \cite{kim2018existence}), or an analysis of nonlinear (e.g., monotone) parabolic type problems, e.g., \cite{Paronetto2013}.   However, the well-posedness of PDEs posed in domains that undergo topological changes during their evolution appears to be essentially unexplored in the existing literature.

To address this apparent gap, we consider a linear parabolic PDE with homogeneous Dirichlet boundary conditions posed in a domain whose evolution can be described by a smooth level set function depending on space and time. If one restricts to the generic  case where the level set function has an isolated nondegenerate critical point (cf. Section~\ref{Sectdomains} for definitions) a complete classification of the possible topological changes is known from the literature. In spatial dimension two, for example, the only scenarios that can occur are creation/vanishing of an island, domain merging and splitting or the creation/vanishing of a hole.     In this paper we study all scenarios that can occur in spatial dimensions two and three.  For most of these scenarios, we prove that a natural weak formulation of the problem is well posed.

The approach taken in this paper extends the standard analysis for cylindrical domains, where a weak solution is defined in the Bochner space $L^2([0,T];H^1_0(\Omega))$ with a time derivative in $L^2([0,T];H^{-1}(\Omega))$. Corresponding extensions to non-cylindrical space--time domains defined by a smooth evolution of a spatial domain $\Omega(t)$ can be found in~\cite{lions1957,alphonse2015abstract,alphonse2023function}. Such a smooth evolution is typically defined through the existence of a smooth (in space and time) diffeomorphism $\Phi_t$ from a reference domain onto $\Omega(t)$. The existence of a smooth (or even continuous) family of diffeomorphisms $\Phi_t$ excludes any change in topology.

We introduce the spaces $H_0$ and $W$, which are analogues of the space $L^2(0,T;H_0^1(\Omega))$ and its subspace consisting of functions whose time derivative belongs to $L^2(0,T;H^{-1}(\Omega))$. These are anisotropic function spaces defined on the space--time domain. 
Instead of relying on a global smooth transformation to a cylindrical space–time domain, we exploit geometric information near the topological singularity to analyze these spaces. We establish key functional-analytic properties of these spaces, including the density of compactly supported smooth functions, which is far from obvious in the presence of topological singularities. This enables us to apply the general Babu\v{s}ka--Banach framework to prove the existence and uniqueness of a weak solution, which satisfies the corresponding \emph{a priori} estimate and solves the original PDE in the sense of distributions. Moreover, any distributional solution to the original problem is  our weak solution.

We note that topological transitions---for example, the merging of two connected components of $\Omega(t)$---may occur in various ways, and the present analysis covers only scenarios described by smooth level set functions with isolated nondegenerate critical points; see Section~\ref{Sectdomains}. All these scenarios can be classified and our analysis covers all these, except that of the creation or vanishing of a hole in a 2D domain and a void inside a 3D domain.

The remainder of the paper is organized as follows. In Section~\ref{Sectdomains}, we recall the level set description of $\Omega(t)$ and make our assumptions precise. Using the Morse lemma, this section provides a classification of possible topological changes and establishes several auxiliary results. Section~\ref{s:model} introduces the model problem and defines the function spaces $H_0$ and $W$; properties of the space $H_0$ are also studied there. Section~\ref{s:spaceW} is devoted to an analysis of the solution space $W$ and contains the main technical results of the paper. Finally, Section~\ref{s:wellP} introduces the weak formulation and demonstrates its well-posedness.

For a domain $\omega \subset \mathbb{R}^n$, we use the notation $(\cdot,\cdot)_\omega$ and $\|\cdot\|_\omega$ for the $L^2(\omega)$ inner product and the corresponding norm, respectively.

\section{Domain evolution} \label{Sectdomains} 
For $t \in [t_0,T]$, consider a time-dependent bounded domain $\Omega(t) \subset \mathbb{R}^d$, where $d = 2,3$, with $\Gamma(t) = \partial\Omega(t)$.

 We assume that there exists a \emph{critical time} $t_c \in [t_0,T]$ such that the evolution of $\Omega(t)$ is smooth before and after $t_c$ in the following sense: 
Let $\mathcal{I}_{-} := (t_0,t_c)$, $\mathcal{I}_{+} := (t_c, T)$, we assume that
 there exist smooth bounded domains $\Omega_\pm^0 \subset \mathbb{R}^d$,  (each consisting of a finite number of connected components), and diffeomorphisms 
\begin{equation}\label{Diff} \Phi_\pm(t): \Omega_\pm^0 \to \Omega(t),\quad t \in \mathcal{I}_\pm,~~ \text{such that}~~\Phi_\pm \in C^\infty(\overline{\Omega_\pm^0} \times \mathcal{I}_\pm; \mathbb{R}^d).
\end{equation}
\rev{The domains $\Omega_-^0$ and $\Omega_+^0$ may, but need not, coincide with $\Omega(t_0)$ and $\Omega(T)$, respectively.}
\begin{figure}
	\centering
	\includegraphics[width=0.85\textwidth]{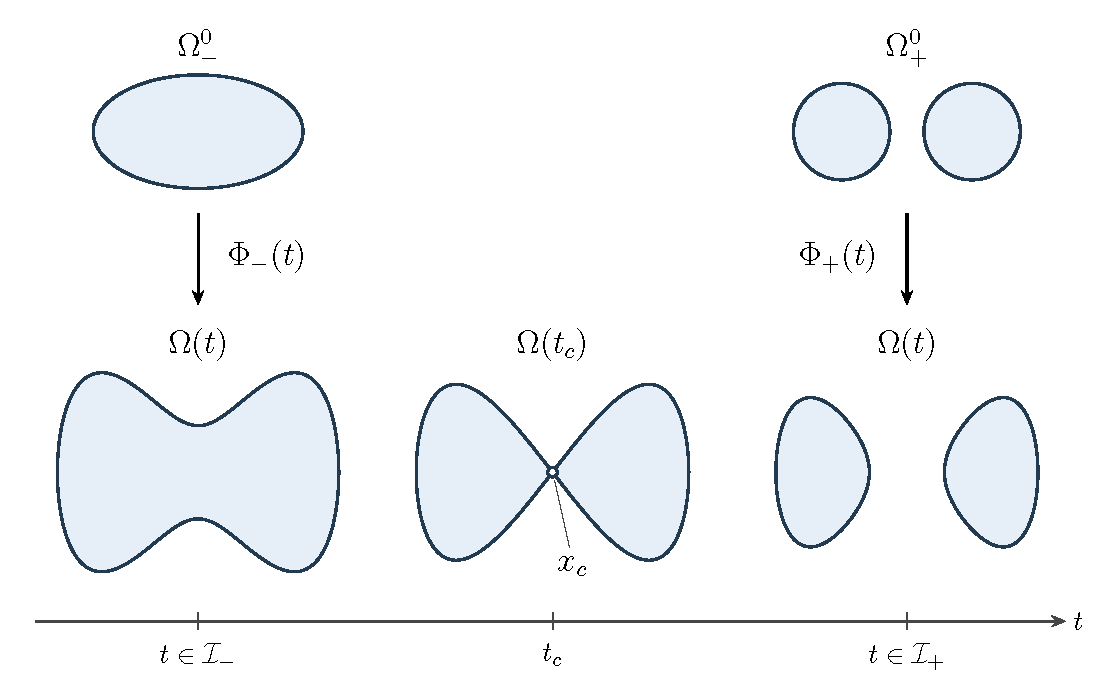}
	\caption{Two-dimensional domain undergoing  splitting and the reference domains in~\eqref{Diff}.}
	\label{fig:reference-domain-splitting}
\end{figure}
 
 Since continuous transformations preserve the topology, $t_c$ is the moment where topology of $\Omega(t)$ may change; see Fig~\ref{fig:reference-domain-splitting}. There can be various  scenarios of topological transitions which fall into this general description  and analysis may heavily depend on further assumptions on the evolution of $\Omega(t)$  in the vicinity of the critical time.
 In this study we restrict to a class of  domains characterized as the subzero levels of a globally smooth level set function $\phi$:
\begin{equation} \label{subzero} 
	\forall ~t \in [t_0,T]:~~\Omega(t)=\{\, x \in \R^d ~|~ \phi(x,t) < 0\,\},\quad \Gamma(t)=\{\, x \in \R^d ~|~ \phi(x,t) = 0\,\},
\end{equation}
and
\begin{equation} \label{eq:smooth} 
	\phi \in C^\infty(\R^d \times [t_0,T]). 
\end{equation}
The analysis below can be extended to less smooth $\phi$. Here,  however, we do not focus on minimal regularity assumptions for the level set function $\phi$. 

For the space--time domain and its `spatial' boundary we use the notations: 
\begin{equation} \label{spatialb}
\cQ = \bigcup\limits_{t \in (t_0,T)} \Omega(t) \times \{t\}\subset \R^{d+1}, \quad  \GQ:= \bigcup\limits_{t \in (t_0,T)} \Gamma(t)\times \{t\}.
\end{equation}

If $|\nabla \phi(x,t)| \geq c_0 > 0$ for all $x \in \Gamma(t)$ and $t \in [t_0,T]$, then, by the implicit function theorem, $\Gamma(t)$ does not undergo topological changes.
To allow for topological changes, we assume an isolated \emph{critical point} $(x_c,t_c) \in \overline\GQ$, i.e.,
\begin{equation} \label{critpoint}
	\nabla \phi(x_c,t_c)=0 \quad\text{and}\quad \nabla \phi(x,t)\neq 0 \quad \text{for all }~(x,t)\in \Gamma_\cQ \setminus (x_c,t_c). 
\end{equation}

\begin{remark}\rm For $\Omega(t)$ given by \eqref{subzero}, \eqref{eq:smooth}, and \eqref{critpoint}, the smooth diffeomorphisms from   \eqref{Diff} can be defined as the flow maps $\Phi_\pm(y,t)=x(t)$,  with $x:\mathcal{I}_{\pm}\to\R^d$ solving 
	   \begin{equation}\label{trajectories}
	   	 x_t=  V(x,t), \quad t\in \mathcal{I}_{\pm},\quad x(\rev{t_\pm^0})=y \in \Omega_{\pm}^0, \quad \text{and}~~
	   	  V=\frac{\rev{-}\phi_t \nabla\phi}{|\nabla\phi|^2+\phi^2}.
	   \end{equation}
Here we take $\Omega_{-}^0=\Omega(t_0)$, $\Omega_{+}^0=\Omega(T)$, and \rev{$t_{-}^0=t_0$, $t_{+}^0=T$}.
Note that for $x\in\Gamma(t)$, $(x,t)\neq (x_c,t_c)$, the normal velocity of the domain boundary is  
\begin{equation}\label{Vg}
	V_\Gamma=-\frac{\partial \phi}{\partial t} \frac{1}{|\nabla \phi|}\bn,
\end{equation}
 where $\bn= \frac{\nabla \phi}{|\nabla \phi|}$  is the outward pointing normal vector. Therefore, at an isolated critical point the normal velocity is unbounded if $\phi_t(0,0)\neq0$.
%\begin{equation} 
%	\label{eq8} \lim_{\small \begin{matrix} (x,t) \in \GQ^-  \\ (x,t) \to (x_c,t_c) \end{matrix}} |V_\Gamma (x,t)|= \infty. 
%\end{equation}
\end{remark}
\smallskip

The critical point is called \emph{nondegenerate}
 if
\begin{equation} \label{critpoint2} 
	\det \big( \nabla^2 \phi(x_c,t_c)\big) \neq 0.
\end{equation}
In this case, certain topological transitions can be classified via Morse theory~\cite{milnor1963morse}. We make use of the following parameter-dependent version of the Morse lemma proved in \cite{Laurain2018}.

\smallskip \begin{lemma}\label{L:Morse} Consider a nondegenerate critical point $(x_c,t_c)$ on the zero level of $\phi$. Without loss of generality, assume $(x_c,t_c)=(0,0)$. Then there exists a neighborhood $\widehat X= X \times (-\delta,\delta)$ of $(0,0)$ in $\R^{d+1}$ and a map $\psi:\, \widehat X \to \R^d$ such that: $\psi(0,0)=0$, $\psi \in C^\infty(\widehat X)$, $\psi(\cdot, t)$ is a diffeomorphism from $X$ onto $\psi(X,t)$ for each $t \in (-\delta,\delta)$. Furthermore, the following normal form holds:
	\begin{equation} \label{normalform}
		\phi\big(\psi(x,t),t\big)= - \sum_{i=1}^q x_i^2 + \sum_{i=q+1}^d x_i^2 + v(t), 
	\end{equation}
	with $0 \leq q \leq d$ and $v: (-\delta,\delta) \to \R$  a smooth map such that $v(0)=0$, $v'(0)=\frac{\partial \phi}{\partial t}(0,0)$. 
\end{lemma}

If $\frac{\partial \phi}{\partial t}(0,0)\neq0$, then due to the inverse function theorem 
there exists $0< \hat \delta \leq  \delta$ sufficiently small such that  for $t \in (-\hat \delta,\hat \delta)$ there is a smooth inverse map $v^{-1}$, i.e., $v^{-1}(v(t))=t$ for $t \in (-\hat \delta,\hat \delta)$. Using the  smooth diffeomorphism  
\begin{equation} \label{transform2}
\Psi(x,t) = (\psi(x,v^{-1}(t)),v^{-1}(t)) \in C^\infty(\widehat X\times (-\hat \delta,\hat \delta)),
\end{equation}
we obtain  new coordinates in which  the level set function takes the form
\begin{equation} \label{normalform2}
	\phi\circ\Psi= - \sum_{i=1}^q x_i^2 + \sum_{i=q+1}^d x_i^2 + \text{s}\,t,  \qquad  \text{s}=\sign \big(\tfrac{\partial \phi}{\partial t}(0,0)\big).
\end{equation}

\subsection{A classification result} \label{secclassification}
Further we are interested in critical points satisfying
\footnote{\rev{
		The inverse function theorem applied to the map $(x,t)\mapsto(\nabla_x\phi(x,t),\phi(x,t))$ implies that any nondegenerate critical point satisfying \eqref{critpoint3} is locally isolated on the space--time boundary $\Gamma_{\mathcal Q}$. To simplify the presentation, we further restrict attention to a single such point. This motivates why in \eqref{critpoint} we assume $\nabla \phi(x,t) \neq 0$ for $(x,t) \neq (x_c,t_c)$.}}
\begin{equation} \label{critpoint3} 
	\frac{\partial \phi}{\partial t}(x_c,t_c) \neq 0. 
\end{equation}

Using $\psi(x_c,t_c)=x_c$ and $\nabla \phi(x_c,t_c)=0$ we compute
\[
\begin{split}
	\nabla^2\phi\big(\psi(x_c,t_c),t_c\big) &= [\nabla\psi(x_c,t_c)]^T\nabla^2\phi(x_c,t_c) \nabla\psi(x_c,t_c) + \sum_{i=1}^d\frac{\partial\phi}{\partial x_i}(x_c,t_c)\nabla^2\psi_i(x_c,t_c)\\
	&= 
	[\nabla\psi(x_c,t_c)]^T\nabla^2\phi(x_c,t_c) \nabla\psi(x_c,t_c).
\end{split}
\]
Since $\nabla\psi(x_c,t_c)\in \R^{d\times d}$ is non-singular, the Sylvester law of inertia and \eqref{normalform} give
\begin{equation}\label{Inertia}
	\text{Inertia}\big(\nabla^2\phi(x_c,t_c)\big) = \{d-q,q,0\}.
\end{equation}
Recalling that the evolving domains are characterized by $\phi < 0$, thanks to the observation in \eqref{Inertia} and the normal form \eqref{normalform}  we can classify nondegenerate critical points $(x_c,t_c)$ as follows.   Let the spectrum of $\nabla^2\phi(x_c,t_c)$ be
\[
\lambda_1\le\dots\le \lambda_d,\quad \lambda_i\in \sigma \big(\nabla^2\phi(x_c,t_c)\big). 
\]
For nondegenerate critical points all eigenvalues are non-zero and we obtain the following

\smallskip 

\begin{classification}[{Classification of nondegenerate critical points}]\rm \label{Class}
%\noindent\texttt{Classification of nondegenerate critical points}\\
\newline The case $d=2$:
\begin{itemize}
	\item[\textrm{2a.}] $0<\lambda_1\le\lambda_2$,   $\frac{\partial \phi}{\partial t}(x_c,t_c) > 0$:  {vanishing of an island}, \\[0.2ex]
	\hspace*{1.9cm} $\frac{\partial \phi}{\partial t}(x_c,t_c) < 0$:  {creation of an island}.\\[-1.5ex]
	\item[\textrm{2b.}] $\lambda_1<0<\lambda_2$,  $\frac{\partial \phi}{\partial t}(x_c,t_c) > 0$:   {domain splitting}, \\[0.2ex]
	\hspace*{1.9cm} $\frac{\partial \phi}{\partial t}(x_c,t_c) < 0$:  {domain merging}.\\[-1.5ex]
	\item[\textrm{2c.}] $\lambda_1\le\lambda_2<0$,  $\frac{\partial \phi}{\partial t}(x_c,t_c) > 0$:   {creation of a hole}, \\[0.2ex]
	\hspace*{1.9cm} $\frac{\partial \phi}{\partial t}(x_c,t_c) < 0$:  {vanishing of a hole}.\\[-1.5ex]
\end{itemize}
The case $d=3$:\smallskip
\begin{itemize}
	\item[\textrm{3a.}] $0<\lambda_1\le\lambda_2\le\lambda_3$,   $\frac{\partial \phi}{\partial t}(x_c,t_c) > 0$:  {vanishing of an island}, \\[0.2ex]
	\hspace*{2.8cm} $\frac{\partial \phi}{\partial t}(x_c,t_c) < 0$:  {creation of an island}.\\[-1.5ex]
	\item[\textrm{3b.}] $\lambda_1<0<\lambda_2\le\lambda_3$,  $\frac{\partial \phi}{\partial t}(x_c,t_c) > 0$:   {domain splitting}, \\[0.2ex]
	\hspace*{2.8cm} $\frac{\partial \phi}{\partial t}(x_c,t_c) < 0$:  {domain merging}.\\[-1.5ex]
	\item[\textrm{3c.}]  $\lambda_1\le\lambda_2<0<\lambda_3$,  $\frac{\partial \phi}{\partial t}(x_c,t_c) > 0$:  {creation of a hole through the domain}, \\[0.2ex]
	\hspace*{2.8cm} $\frac{\partial \phi}{\partial t}(x_c,t_c) < 0$:  {vanishing of a hole through the domain}.\\[-1.5ex]
	\item[\textrm{3d.}] $\lambda_1\le\lambda_2\le\lambda_3<0$,  $\frac{\partial \phi}{\partial t}(x_c,t_c) > 0$:  {creation of an interior void}, \\[0.2ex]
	\hspace*{2.8cm} $\frac{\partial \phi}{\partial t}(x_c,t_c) < 0$:  {vanishing of an interior void}.\\[-1.5ex]
\end{itemize}   
\end{classification}   
For $d=2$, a domain splitting scenario is illustrated in Figure~\ref{fig1}.
\smallskip

\begin{figure}[ht!] 
	\begin{center} 
		\includegraphics[width=0.45\textwidth]{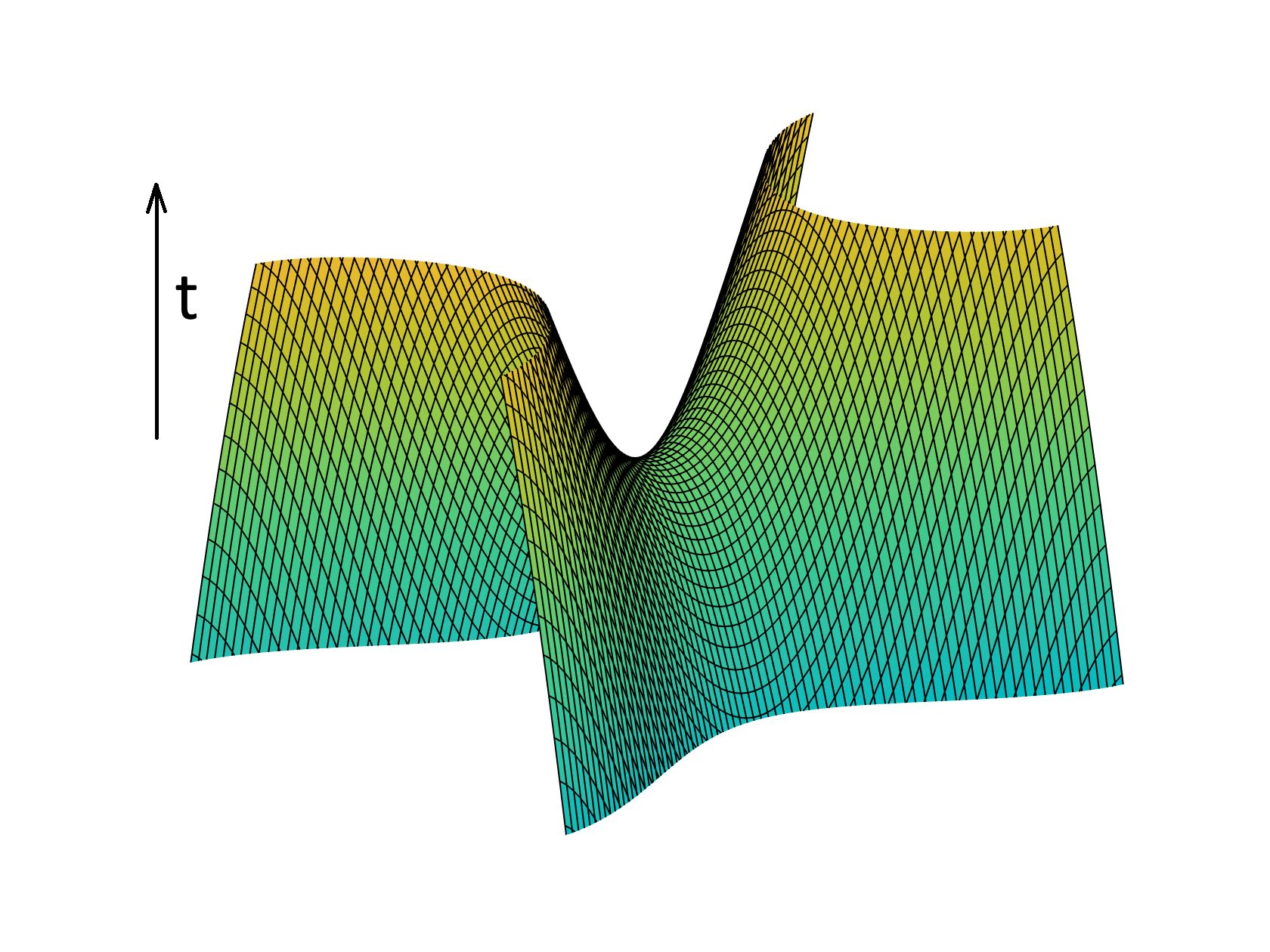} 
		\includegraphics[width=0.4\textwidth]{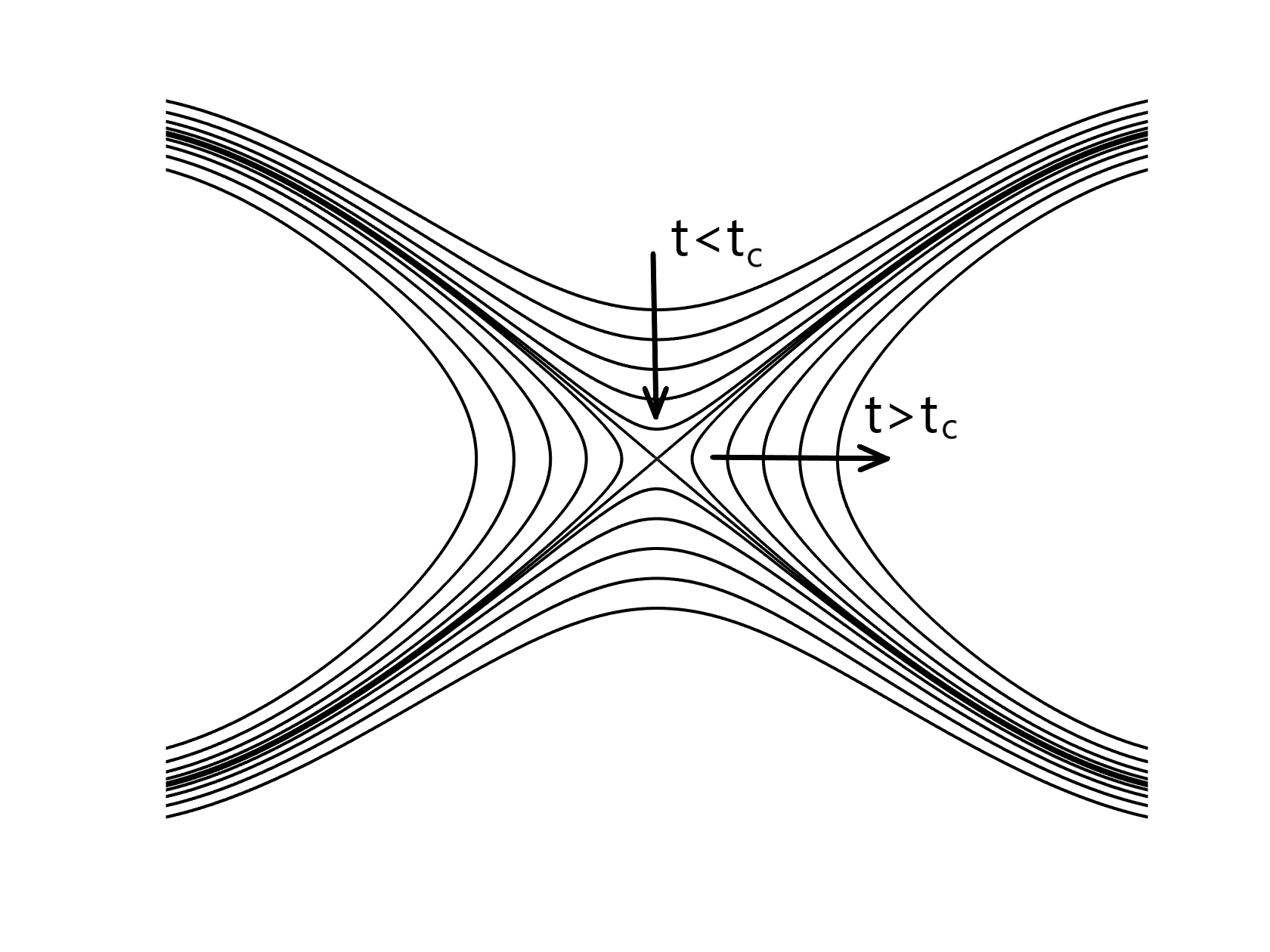} 
	\end{center} 
	\vskip-1.5em 
	\caption{$d=2$. A splitting scenario with $\lambda_1<0<\lambda_2$,  $\frac{\partial \phi}{\partial t}(x_c,t_c) > 0$. The left plot shows $\GQ$ near $(x_c,t_c)$, while the right shows snapshots of $\Gamma(t)$ near $(x_c,t_c)$.} \label{fig1} 
\end{figure}

Without loss of generality we shall  assume 
 \[ t_0=-1, \quad T=1, \quad  \text{and}~~ (x_c,t_c)=(0,0).
 \]
This will be our default assumption. The classification cases 2a) and 3a), however, are exceptions \rev{because a connected domain component exists only on one side of the critical time.}
For these cases, we take $t_c$ \rev{to be the initial time for island creation and the final time for island vanishing:}
 \begin{equation} \label{exception} 
 \begin{split}
 t_0=0, \quad T=1, \quad   (x_c,t_c)=(0,0)\qquad&\text{creation of an island},\\
 t_0=-1, \quad T=0, \quad  (x_c,t_c)=(0,0)\qquad&\text{vanishing of an island},
 \end{split}
\end{equation}
without specifying this each time explicitly.

\subsection{Uniform Poincar\'{e} and trace inequalities} 
%\footnote{MO  says:  We need Uniform Poincare to claim the coercitivity of the bilinear form for the heat equation. The trace inequality we need to give meaning to traces and so to the zero Dirichlet b.c. Alternatively would be to define $H_0$ as the closure ob all compactly supported smooth functions in $H$.  However, we then may need to deal with other technical issues, e.g. we use at some point the zero extension of functions from $H_0$. We would need to argue why such extension is in $L^2(0,T; H^1_0(\hat B))$. }
For the time dependent domains $\Omega(t)$ given by the subzero levels of a smooth $\phi$ with an isolated nondegenerate critical point satisfying \eqref{critpoint3}, there are uniform in $t$ bounds on the best constant in the Poincar\'{e} and trace inequalities. This is made precise in this section.  For $t \in [-1,1] \setminus \{0\}$ the domain $\Omega(t)$ has a smooth boundary and the $H_0^1(\Omega(t))$ subspace of $H^1(\Omega(t))$ functions that have zero trace on $\partial \Omega(t)$ is defined in the usual way. 
%For topology scenario C we denote by $H^1_{0,{\rm out}}(\Omega(t))$ the space consisting of $H^1(\Omega(t))$ functions that have zero trace only on the outer boundary. 
\begin{lemma} \label{lemPoincare}
There is a constant $C$ such that for all $t \in [-1,1] \setminus \{0\}$ the following holds:
\begin{align}\label{UPoincare}
	\|u\|_{\Omega(t)}\ & \leq C \|\nabla u\|_{\Omega(t)}\quad \text{for all} ~~u\in H^1_0(\Omega(t)). 
%	\\
%	\|u\|_{\Omega(t)}\ & \leq C \|\nabla u\|_{\Omega(t)}\quad \text{for all} ~~u\in H^1_{0,{\rm  out}}(\Omega(t)) \label{UpoincareA}
\end{align}
\end{lemma}
\begin{proof}
Recall that we assumed the domains $\Omega(t)$, $t \neq t_c$, to be smooth and bounded. It is clear that for all scenarios there is a constant $a$ such that $\Omega(t) \subset [-a,a]^d$ for all $t \in [-1,1]$. We now apply a standard argument for all regular times $t \neq 0$.  Take $u \in H^1_0(\Omega(t))$, $t \neq 0$, and extend it by zero outside $\Omega(t)$. The extended function $\bar u$ is from $H^1_0([-a,a]^d)$ and we have thanks to  Poincar\'{e} inequality in the cube:  
\[\|u\|_{\Omega(t)} = \|\bar u\|_{[-a,a]^d} \leq \sqrt{2} a \|\nabla \bar u\|_{[-a,a]^d}
=\sqrt{2} a \|\nabla u\|_{\Omega(t)}
\] 
which proves the result \eqref{UPoincare}.
\end{proof}
\smallskip

%We also need some uniform estimates for domains defined  by $\phi$ in the normal form \eqref{normalform2}. 
%For the vanishing and creation of  an island we assume  that the critical time is the final and initial time, respectively, cf. \eqref{exception}. %Denote by $\Box$ the unit box in $\R^d$. %and by $S$ the ball of radius 2, both in $\R^d$. 

\begin{lemma} \label{uniformtrace}
For the cases 2a),3a) there is a constant $C$ such that the following trace estimate holds 
\begin{equation}\label{UtraceA}
	\|u\|_{\partial\Omega(t)}^2\le C \big(|t|^{-\frac12}  \|u\|_{\Omega(t)}^2+ |t|^{\frac12}
	\|\nabla u\|_{L^2(\Omega(t))}^2\big) \quad \text{for all}~~u \in C^1(\Omega(t))
\end{equation}
and for all $t \in  [-1,0)$ (vanishing island case) or  for all $t \in  (0,1]$  (creation of an island case).

For the other classification cases  there is a constant $C$ such that for all $t \in  [-1,1]\setminus\{0\}$ the following trace estimate holds
\begin{equation}\label{Utrace}
	\|u\|_{\partial\Omega(t)}\le C \|u\|_{H^1(\Omega(t))} \quad \text{for all}~~u \in C^1(\Omega(t)).
\end{equation}
\end{lemma}
\begin{proof} 
For $|t| \geq \delta$, with a given $\delta >0$, there is a diffeomorphism between $\Omega(t)$ and $\Omega(-1)$ or $\Omega(1)$.
For the domains $\Omega(-1)$ and $\Omega(1)$ the standard trace inequality holds, \rev{which is transported by the smooth diffeomorphism so that} we obtain an estimate as in \eqref{Utrace} with a constant uniform in $t$ provided $|t| \geq \delta$. For $|t| \leq \delta$ and $\delta$ sufficiently small, we can use the normal form of $\phi$ and classification~\ref{Class}.

\textbf{\small Cases 2a), 2c), 3a), 3d).} For the cases 2a) and 3a) \rev{and $|t|\le \delta$ after the change of coordinates} we have a domain $\Omega(t)$ that is a ball with radius equal to $|t|^\frac12$. Now use the standard trace inequality for the unit ball in $\R^d$ and a scaling argument. This then yields the result \eqref{UtraceA}. The cases 2c) and 3d) are straightforward.
%\smallskip

\textbf{\small Cases 2b), 3b), 3c).}
For the analysis of these cases  we use a  different technique.
For $\eps >0$ we define $\Box_\eps:=\{\, x \in \R^d\,|\, |x_i|\leq \eps, ~ 1 \leq i \leq d\,\}$. For $|t|\leq \delta$ we will derive a uniform bound for $\|u\|_{\Gamma(t)\cap \Box_\eps}$ in terms of $\|u\|_{H^1(\Omega(t) \cap \Box_{\tilde \eps})}$ for suitable (sufficiently small) $\delta,\eps, \tilde \eps$. To this end, we consider a subdivision of the (local) boundary $\Gamma(t)\cap \Box_\eps$  into a finite number of boundary subdomains. Assume  $\Gamma_+(t)$ is one of such subdomains.  Assume a \emph{given} $w \in \R^d$, with $|w|=1$, such that  $(\xi,\alpha) \to \xi+\alpha w$, $\xi \in \Gamma_+(t)$, defines local coordinates with $\alpha \in [-\alpha_0,\alpha_0]$, $\alpha_0 >0$, independent of $t$. Furthermore, assume that the following two conditions are satisfied, where $n$ denotes the outward pointing unit normal on $\Gamma_+(t)$:
\begin{align} 
	\inf_{\Gamma_+(t)} |w^T n |  \geq c_0 >0 &\quad \text{for all}~t \in [-\delta,\delta], \label{conditie1} \\
	S_{\alpha_0}(t)  := \{ \xi +\alpha w\,|\, \xi \in \Gamma_+(t), ~\alpha \in [0,\alpha_0]\,\} \subset \Omega(t)\cap \Box_{\eps+\alpha_0}& \quad \text{for all}~t \in [-\delta,\delta], \label{conditie2}
\end{align}
with a constant $c_0$ in \eqref{conditie1} that is uniform in $t \in [-\delta,\delta]$. Using a standard approach (\cite[Section 5.5]{Evans}), one obtains 
\begin{equation} \label{K9} \begin{split}
		\int_{\Gamma_+(t)} u^2 \, ds  & \leq c_1 \int_{\Gamma_+(t)} \int_{0}^{\alpha_0} u^2+ (w \cdot \nabla u)^2  \ d\alpha\, d\xi  \\ &
		\leq c_2 \int_{S_{\alpha_0}} u^2 + |\nabla u|^2 \, dx \leq c_2 \|u\|_{H^1(\Omega(t)\cap \Box_{\eps+\alpha_0})}^2,
\end{split} \end{equation}
with constants $c_1$, $c_2$ independent of $t$. This yields the uniform trace estimate \eqref{Utrace} for $\Gamma_+(t)$. We apply this approach for the cases Cases 2b), 3b), 3c).

\textbf{\small Case 2b).} We consider this case with $\phi_t(0,0) >0$, i.e., the splitting scenario. The merging scenario is also covered by replacing $t$ with $-t$. We need to derive a uniform trace estimate for $\Gamma(t)=\{\, (x_1,x_2)\,|\, \phi(x,t):=-x_1^2 + x_2^2 + t =0\,\}$ in a neighborhood of the space-time origin. First consider  the part of $\Gamma(t)\cap \Box_\eps$ in the first quadrant, $\Gamma_+(t):=\{\, x \in \Gamma(t)\cap \Box_\eps\,|\, x_1 \geq 0, ~x_2 \geq 0\,\}$.
  By symmetry we then also have the same uniform trace estimate for the parts of $\Gamma(t)$ in the other quadrants. It remains to verify the conditions \eqref{conditie1}-\eqref{conditie2} for suitable $w$. Take $w:=\frac{1}{\sqrt{2}} (1, -1)^T$. Note that $n=|x|^{-1}(-x_1,x_2)^T$. %Furthermore, for $x \in \Gamma_+(t)$ we have $x_i \geq 0$.  Note that for $t \leq 0$ we have $x_1^2 = x_2^2 +t \leq x_2^2$, hence, $|x|\leq \sqrt{2}\, x_2$, and  for $t \geq 0$ we have $x_2^2 = x_1^2 -t \leq x_1^2$, hence, $|x|\leq \sqrt{2}\, x_1$, Using this we obtain, 
  For $x\in \Gamma_+(t)$ straightforward computations show:
\[
 |n(x)^T w| = \frac{1}{\sqrt{2}} \frac{x_1+x_2}{|x|}  \geq \frac12,
\]
and thus  \eqref{conditie1} is satisfied. One checks that on $\Gamma_+(t)$ we have $w \cdot \nabla \phi \leq 0$ and $w^T\nabla^2 \phi w=0$. Thus we obtain, for $\xi \in \Gamma_+(t)$:
\begin{equation} \label{Taylor}
  \phi(\xi+\alpha w)=\phi(\xi)+\alpha w \cdot \nabla \phi(\xi) + \tfrac{\alpha^2}2 w^T \nabla^2 \phi w \leq 0.
\end{equation}
Hence, $S_{\alpha_0}(t) \subset \Omega(t)$ for all $\alpha_0 > 0$. Finally note that for $x= \xi +\alpha w \in S_{\alpha_0}(t) $ we have $|x_i| \leq |\xi_i|+\alpha_0 |w_i| \leq \eps +\alpha_0$. Thus we also verified \eqref{conditie2}. \\
%For the case $t \geq 0$ we use the same approach and consider $\Gamma_+(t):=\{\, x \in \Gamma(t)\cap \Box_\eps\,|\, %x_1 \geq 0\,\}$ (by symmetry we then also have the result for the branch with $x_1 \leq 0$). We take $w=(1,0)^T$. On $\Gamma_+(t)$ we then have, using $x_2^2= x_1^2-t \leq x_1^2$, 
%\[ 
% |n(x)^T w| =  \frac{x_1}{|x|}\geq \frac{1}{\sqrt{2}},
%\]
%and thus \eqref{conditie1} holds. Furthermore, on $\Gamma_+(t)$ we have $w \cdot \nabla \phi \leq 0$ and $w^T\nabla^2 \phi w \leq 0$. As in \eqref{Taylor} we obtain $S_{\alpha_0}(t) \subset \Omega(t)$ for all $\alpha_0 > 0$, and with the same arguments as above we conclude that the uniform trace estimate \eqref{Utrace} holds. \\
%
%
\textbf{\small Case 3b).} It suffices to consider the splitting scenario with $\Gamma(t)= \{\, x\,|\, \phi(x):=-x_1^2+x_2^2 +x_3^2+t=0 \,\}$, since then the merging scenario is treated by replacing $t$ by $-t$.  Note that $n=|x|^{-1}(-x_1,x_2,x_3)^T$. Points $x \in \Gamma(t)$ can be parameterized by $(x_1,\psi) \to (x_1, r \cos \psi, r \sin \psi)^T=x(x_1,\psi)$, with $r:=\sqrt{x_1^2-t}$, $x_1^2 \geq t$, $\psi \in [0,2\pi)$.  We take $\Gamma_+(t):=\{\, x(x_1,\psi)\in \Gamma(t)\cap \Box_\eps\,|\,   \tfrac{\pi}{4} \leq  \psi \leq \tfrac{3\pi}{4}~\text{and}~x_1 \geq 0\,\}$. The result for other parts of $\Gamma(t)\cap \Box_\eps$ are then obtained by a symmetry argument (note that $\Gamma(t)$ is a surface of revolution around the $x_1$-axis). We take $w:=\frac{1}{\sqrt{2}}(1,0,-1)^T$. On $\Gamma_+(t)$ we have $x_1 \geq 0$ and $x_3=r \sin \psi \geq \tfrac12  \sqrt{2}\, r$ and we obtain
\begin{equation} \label{A44}
  |n(x)^T w| =\frac{1}{\sqrt{2}} \frac{x_1+x_3}{|x|} \geq \frac{1}{\sqrt{2}}\frac{x_1+\frac12 \sqrt{2}\, r}{\sqrt{x_1^2+r^2}}\geq \frac12\frac{x_1+r}{\sqrt{x_1^2+r^2}} \geq \frac12,
\end{equation}
and thus \eqref{conditie1} holds. Furthermore, on $\Gamma_+(t)$ we have $w \cdot \nabla \phi \leq 0$ and $w^T\nabla^2 \phi w = 0$. One can apply the same arguments as in case 2b) to verify that \eqref{conditie2} is satisfied and hence the uniform trace estimate \eqref{Utrace} holds. 
%For the case $t \geq 0$ we  take $\Gamma_+(t):=\{\, x \in \Gamma(t)\cap \Box_\eps\,|\, x_1 \geq 0\,\}$ (by symmetry we then also have the result for the branch with $x_1 \leq 0$). We take $w=(1,0,0)^T$. With the same arguments as used in case 2b) for $t\geq 0$ one concludes that the conditions \eqref{conditie1}-\eqref{conditie2} are satisfied. 
\\
\textbf{\small Case 3c).} It suffices to consider $\Gamma(t)= \{\, x\,|\, \phi(x):=-x_1^2-x_2^2 +x_3^2+t=0 \,\}$. Note that $n=|x|^{-1}(-x_1,-x_2,x_3)^T$. Points $x \in \Gamma(t)$ can be parameterized by $(x_3,\psi) \to (r \cos \psi, r \sin \psi, x_3)^T=x(x_3,\psi)$, with $r:=\sqrt{x_3^2+t}$, $x_3^2 \geq -t$, $\psi \in [0,2\pi)$.  We take $\Gamma_+(t):=\{\, x(x_3,\psi)\in \Gamma(t)\cap \Box_\eps\,|\,   \tfrac{\pi}{4} \leq  \psi \leq \tfrac{3\pi}{4}~\text{and}~x_3 \geq 0\,\}$. The result for other parts of $\Gamma(t)\cap \Box_\eps$ are then obtained by a symmetry argument (note that $\Gamma(t)$ is a surface of revolution around the $x_3$-axis). We take $w:=\frac{1}{\sqrt{2}}(0,1,-1)^T$. On $\Gamma_+(t)$ we have $x_3 \geq 0$ and $x_2=r \sin \psi \geq \tfrac12  \sqrt{2}\, r$ and with arguments as in \eqref{A44} we obtain $|n(x)^T w| \geq \tfrac12$ for $x \in \Gamma_+(t)$ and thus \eqref{conditie1} holds. Due to $w \cdot \nabla \phi \leq 0$ and $w^T\nabla^2 \phi w = 0$ on $\Gamma_+(t)$ the same arguments as above apply again.
\end{proof}

\section{Model problem and space-time function spaces}\label{s:model}
We assume  that an evolving domain $\Omega(t)$, $t \in [-1,1]$, as specified in Section~\ref{Sectdomains} is given and consider the  heat equation as a model problem:
\begin{equation} \label{transport}
%		\frac{\partial u}{\partial t} + \div(u\bw) - \alpha \Delta u&=0\quad\text{on}~~\Omega(t), ~~t\in (-1,1],
				\frac{\partial u}{\partial t}  - \Delta u=f\quad\text{on}~~\Omega(t), ~~t\in (-1,1],
\end{equation}
with initial condition $u(x,-1)=u_0(x)$ for $x \in \Omega(-1)$ (except for  the classification case of an island creation).  We assume   homogeneous Dirichlet boundary condition
\begin{equation} \label{Dirichlet}
		u = 0 \quad \text{on}~~\partial \Omega(t),~~t\in (-1,1].
\end{equation}
%or Neumann's boundary conditions on $\partial \Omega(t)$,
%\begin{equation} \label{Neumann}
%	\nabla u \cdot \bn = 0 \quad \text{on}~~\partial \Omega(t),~~t\in (-1,1].
%\end{equation}

A sufficiently regular solution to \eqref{transport}-- \eqref{Dirichlet} %or \eqref{transport}, \eqref{Neumann}
 satisfies the integral identity 
\begin{equation}\label{integral}
\left(u_t, v\right)_\cQ + \left(\nabla u, \nabla v\right)_\cQ =\left(f, v\right)_\cQ
\end{equation}
for any sufficiently smooth $v$ vanishing on $\Gamma_\cQ$.   The integral equality suggests functions from $L^2(\cQ)$ with space gradients in $L^2(\cQ)$ to form an appropriate test space and to consider  solutions $u$ from the same space, for which  time derivative $u_t$ exists in a suitable weak sense.

We introduce such spaces below  and derive  relevant properties. In section~\ref{s:wellP} we use these spaces to show well-posedness of a more general parabolic problem.

\subsection{Function spaces}
%In this section we introduce natural function spaces that are suitable for a \emph{space-time} weak variational formulation of the heat equation \eqref{transport}, cf. also \eqref{integral}. 

For the space--time domain $\cQ \subset \R^{d+1}$ let $C_c^\infty(\cQ)$ be the space of infinitely differentiable functions  with compact support in $\cQ$.
Consider a space consisting of functions from $L^2(\cQ)$ for which  the weak first partial derivatives with respect to the space variables exist in $L^2(\cQ)$:
\begin{equation} \label{defH}
   H:= \left\{ u \in L^2(\cQ)~|~\partial_i u\in L^2(\cQ), ~1 \leq i \leq d \,\right\}.
\end{equation}
%where $\partial_iu$ is the weak partial derivative. %, i.e., $\int_{\cQ} u \frac{\partial \phi}{\partial x_i}\, dq = - \int_{\cQ} \partial_i u \, \phi \, dq$ for all $\phi \in C_c^\infty(\cQ)$. 
Endowed with the scalar product
\[
  (u,v)_H= (u,v)_{\cQ}+ (\nabla u, \nabla v)_{\cQ}, \quad u,v\in H,
\]
this is a Hilbert space. In the literature we did not find a result showing that functions from this anisotropic Sobolev space can be approximated by functions that are smooth up to the boundary of $\cQ$. The next theorem states that the expected density result indeed holds. 
\smallskip

\begin{theorem} \label{densityH}
 The space $C^\infty(\overline \cQ)$ is dense in $H$.
\end{theorem}
\begin{proof}
 The arguments used in standard references, e.g. \cite{Adams2003,Evans}, for proving density of globally smooth functions in the standard Sobolev space $W^{m,p}(\Omega)$ also apply to this anisotropic case. Details are given in Appendix~\ref{AppA}. 
\end{proof}
\smallskip

Recall that $V_\Gamma$ is the normal velocity of $\Gamma(t)$, cf. \eqref{Vg}. Then, for any $u\in C^\infty(\overline \cQ)$, it holds that
\begin{equation}\label{TraceEsta}
\int_{\Gamma_\cQ}(1+|V_\Gamma|^2)^{-\frac12}u^2\,dq= 
\int_{-1}^1\int_{\Gamma(t)}u^2\,ds\,dt.
\end{equation}
To define a trace on $\Gamma_\cQ$  for functions from $H$, we introduce the weight function
\[
\omega=\left\{\begin{split}
|t|^{\frac12}(1+|V_\Gamma|^2)^{-\frac12} &\quad \text{for cases 2a), 3a),}\\
(1+|V_\Gamma|^2)^{-\frac12} &\quad  \text{for other cases}.
\end{split} \right.
\]
The corresponding  weighted $L^2$-space is denoted by $L^2_{\omega}(\Gamma_\cQ)$.
From the definition of $\omega$ and the uniform trace inequalities \eqref{UtraceA}--\eqref{Utrace} it follows that 
\begin{equation}\label{TraceEst}
	\|u\|_{L^2_{\omega}(\Gamma_\cQ)} \le C \|u\|_H \quad \text{for all}~u\in C^\infty(\overline \cQ)
\end{equation}
holds. Owing to the density result in Theorem~\ref{densityH} and a standard continuity argument, the trace operator $u|_{\Gamma_\cQ}:\, H \to L^2_{\omega}(\Gamma_\cQ)$ is well defined.
%$ for $u\in H$ as a linear bounded 
%operator from $H$ into the weighted $L^2$-space $L^2_{\omega}(\Gamma_\cQ)$.

We now can define the space of functions from $H$ vanishing on the space boundary of $Q$ in the sense of their traces,
\[
H_0= \{u\in H\,:\, u|_{\Gamma_\cQ}=0\}.
\]

%Below we show that $C^\infty_c(\cQ)$ is dense in $H_0$. 
\begin{theorem} \label{densityH0}
	The space $C^\infty_c(\cQ)$ is dense in $H_0$.
\end{theorem}
\begin{proof}
Fix some $\eps \in (0,1)$, \rev{for all classification cases except 2a) and 3a)} define $\cQ_\eps=\bigcup\limits_{t\in(-1,-\eps)\cup (\eps,1)}\hskip-2ex \Omega(t)\times\{t\}$, 
\rev{otherwise let $\cQ_\eps=\hskip-1ex\bigcup\limits_{t\in(-1,-\eps)}\hskip-2ex \Omega(t)\times\{t\}$ or $\cQ_\eps=\bigcup\limits_{t\in(\eps,1)}\hskip-1ex \Omega(t)\times\{t\}$ for vanishing or creation an island, respectively}. 
For  an arbitrary  $u\in H_0$ define $u_\eps=\left\{\begin{array}{ll}
		u & \text{in } \cQ_\eps,\\
		0 & \text{in } \cQ\setminus\cQ_\eps.
	\end{array}\right.$
We have $u_\eps\in H_0$ and $\|u-u_\eps\|_H\to0$ for $\eps\to0$. For any fixed $\eps>0$ the domain $\cQ_\eps$ is the union of two subdomains $\cQ_\eps^{\pm}$, each corresponding to a smooth evolution of $\Omega(t)$, cf. \eqref{Diff}. For example,	$\cQ_\eps^+=\bigcup_{t\in(\eps,1)}\Omega(t)\times\{t\}$, where $\Omega(t)=\Phi^+(t)(\Omega(1))$ and $\Phi^+\in C^\infty(\overline{\Omega(1)}\times[\eps,1])$. By standard arguments based on cutting off and mollifying, see, e.g. \cite{lions1957}, one shows that there is a function $\psi_\eps\in C^\infty_c(\cQ_\eps)$  such that $\|u_\eps-\psi_\eps\|_{L^2(\cQ_\eps)}+\|\nabla(u_\eps-\psi_\eps)\|_{L^2(\cQ_\eps)}^2\le \eps$. Extending $\psi_\eps$  by zero to  $\cQ$ and letting $\eps\to 0$ yields by the triangle inequality $\|u-\psi_\eps\|_H\to0$ and so proves the theorem.\quad
\end{proof}
\smallskip

Besides the (test) space $H_0$ we also need a suitable solution space containing elements from $H_0$ that have  a well-defined  weak time derivative. 
For $u\in H_0$, following a standard approach we consider $\partial_t u$ as a distribution: 
\[
\langle \partial_t u,\xi\rangle:= -\int_{\cQ} u\,\partial_t \xi\, dq,\quad \text{for}~\xi\in C^\infty_c(\cQ).
\]
We restrict to $u \in H_0$ for which  
\[
\|\partial_t u\|_{H^{-1}_0}:=\sup_{0\neq\xi\in C^\infty_c(\cQ)}\frac{|\langle \partial_t u,\xi\rangle|}{ \|\xi\|_H} <\infty
\]
holds. In this case, $\partial_t u$ is uniquely  extended to a linear bounded functional on $H_0$. 

As the solution space for our problem we propose
\begin{equation} \label{defW}
W = \left\{u\in H_0\,|\, \partial_t u \in H_0^{-1}\right\}, \quad\text{with}~\|u\|_W=\left(\|u\|_H^2+\|\partial_t u\|_{H^{-1}_0}^2\right)^{\frac12}.
\end{equation} 
In the next section we shall see that smooth functions vanishing on $\Gamma_\cQ$ are dense in $W$ for certain classes of the classification from Section~\ref{secclassification}. We also show a Lions-Magenes type result for $W$. 

\section{Properties of the solution space $W$} \label{s:spaceW}
In this section, we derive useful properties of the solution space $W$.

To prove the density of smooth functions in $W$, we adapt an approach from \cite{lions1957}. 
For this purpose, we introduce a cut-off function 
$h \in C^\infty([0,\infty))$ such that $h=0$ on $[0,1)$, $h=1$ on $[3,\infty)$, 
$0 \le h \le 1$, and $|h'| \le 1$ everywhere. 
Define $h_\varepsilon(r)=h(r/\varepsilon)$ and
\begin{equation} \label{cutoff}
	\theta_\varepsilon(x,t)
	=
	h_\varepsilon\big(-\phi(x,t)\big),
	\qquad (x,t)\in \cQ .
\end{equation}

Note that $\theta_\varepsilon$ cuts off near the spatial boundary 
$\Gamma_\cQ \subset \mathbb{R}^{d+1}$.
A key property of $\theta_\varepsilon$ is that it is smooth and vanishes in a strip 
(in $\mathbb{R}^{d+1}$) aligned with $\Gamma_\cQ$, whose \textit{space--time} width is 
proportional to $\varepsilon$: Using the fact that $|\nabla_{(x,t)} \phi| \ge c > 0$ on $\Gamma_\cQ$, 
it follows that there exist strictly positive constants $c_0$ and $c_1$ such that, 
for $\varepsilon$ sufficiently small,
\begin{equation} \label{distance}
	\operatorname{dist}(\operatorname{supp}(\theta_\varepsilon), \Gamma_\cQ) 
	\ge c_0 \varepsilon,
	\qquad
	\operatorname{dist}(\operatorname{supp}(1-\theta_\varepsilon), \Gamma_\cQ) 
	\le c_1 \varepsilon .
\end{equation}

Let $u \in H_0$ and define
\[
u_\varepsilon := \theta_\varepsilon u .
\]

The analysis in the remainder of this section proceeds in the following main steps:
\begin{itemize}
	\item[a)] (Lemma~\ref{lemA}) 
	For $u \in H_0$, we show (for certain classification cases) that 
	$\lim\limits_{\eps \to 0} u_\varepsilon = u$ in $H$. 
	In Remark~\ref{RemdiscussionHoles} we explain, why we are not able to derive such a result for the cases 2c) and 3d). 
	
	\item[b)] (Lemma~\ref{L2}) 
	For $u \in W$, we use the result from a) to show that there exists a sequence 
	$(w_n)_{n \ge 1}$ with $w_n$ vanishing on $\Gamma_\cQ$ and 
	$\lim\limits_{n \to \infty} w_n = u$ in $W$.
	
	\item[c)] (Lemma~\ref{L3}) 
	Based on the result in b), we extend functions $u \in W$ by zero outside $\cQ$ 
	and apply a density result for smooth functions in the cylindrical case. 
	This yields the density of smooth functions in $W$.
	
	\item[d)] (Lemma~\ref{L4}) 
	We prove a Lions--Magenes-type result that gives meaning to traces of functions 
	from $W$ on $\Omega(t)$ as elements of $L^2(\Omega(t))$.
\end{itemize}
\smallskip

Concerning the rather technical proof of Lemma~\ref{lemA}, we note the following.
The crucial point in the analysis is to show that
$
\lim\limits_{\eps \to 0} \|(\nabla \theta_\varepsilon) u\|_{\cQ} = 0.
$
On the domain $\cQ \setminus \cO_0$, where $\cO_0$ is a fixed sufficiently small 
space--time neighborhood of the critical point $(0,0)$, we can use arguments that 
are essentially the same as those in \cite{lions1957}. 
On the domain $\cQ \cap \cO_0$, we exploit the explicit structure of the space--time 
singularity provided by Lemma~\ref{L:Morse} and \eqref{normalform2}.
For the latter analysis, a one-dimensional Hardy estimate, presented in the next lemma, is useful.
\begin{lemma} \label{1DPoincare}
Take $0\le a<b$ and $u \in H^1(a,b)$ with $u(b)=0$. Then for $p\geq 0$, the estimate
\begin{equation} \label{ResP1}
 \int_a^b \left(\frac{u}{b-x}\right)^2 x^p \, dx \leq C_p \int_a^b (u')^2  x^p\, dx
\end{equation}
holds with a constant $C_p <\infty $ independent of $a,b$ and $u$. 
\end{lemma}
\begin{proof} By the change of the variables $x=bx'$, we see that \eqref{ResP1}
is equivalent to 	
\begin{equation} \label{ResP2}
 \int_a^1 \left(\frac{u}{1-x}\right)^2 x^p \, dx \leq C_p \int_a^1 (u')^2  x^p\, dx.
\end{equation}
for $u \in H^1(a,1)$ with $u(1)=0$ and \rev{re-scaled} $a\in [0,1)$.
For $p=0$ this a standard Hardy estimate. Take $p \geq 0$, let $\alpha=\max\{\tfrac{1}{2},a\}$.
% and split
%\begin{equation} \label{ResP3}
% \int_0^1 \left(\frac{u}{1-x}\right)^2 x^p \, dx = \int_0^\frac12 \left(\frac{u}{1-x}\right)^2 x^p \, dx + \int_\frac12^1 \left(\frac{u}{1-x}\right)^2 x^p \, dx.
%\end{equation}
For the integral over $[\alpha,1]$ we use the Hardy inequality on that interval and obtain
\begin{equation} \label{ResP4}
\begin{split}
\int_\alpha^1 \left(\frac{u}{1-x}\right)^2 x^p \, dx & \leq \int_\alpha^1 \left(\frac{u}{1-x}\right)^2  \, dx \leq
\tilde c_H \int_\alpha^1 (u')^2  \, dx  \\ & \leq \tilde c_H 2^p \int_\alpha^1 (u')^2 x^p  \, dx \leq \tilde c_H 2^p \int_a^1 (u')^2 x^p  \, dx.
\end{split} \end{equation}
If $a<\tfrac12$, we have another integral over $(a,\tfrac12)$. We first note
\begin{equation} \label{ResP5}
\int_a^{\tfrac12} \left(\frac{u}{1-x}\right)^2 x^p \, dx \leq 4 \int_a^1 u^2 x^p \, dx
\end{equation}
and also
\begin{align*}
 \int_a^1 u^2 x^p \, dx &  = - \frac{a^{p+1}}{p+1}u^2(a)  - \frac{2}{p+1} \int_a^1 u u' x^{p+1}\, dx \\ 
 & \leq  \frac{2}{p+1} \left(\int_a^1 u^2 x^p\, dx\right)^\frac12 \left(\int_a^1 (u')^2 x^p\, dx\right)^\frac12 \\
 & \leq  \frac12 \int_a^1 u^2 x^p\, dx + \frac{2}{(p+1)^2}\int_a^1 (u')^2 x^p\, dx.
\end{align*}
Shifting the integral of $u^2 x^p$ to the left
%, noting that
%$a^{p+1}u^2(a)\le a^{p+1}\int_a^1(u')^2 dx \le \int_a^1 (u')^2 x^p\, dx$,
  and using  
this in \eqref{ResP5} completes the proof of \eqref{ResP2}.
\end{proof}

\smallskip
\begin{lemma} \label{lemA} Consider a space-time domain with critical points as in  classification \ref{Class}, except for the cases 2c) and 3d). 
 For $u \in H_0$ we have $\lim\limits_{\eps \to 0}\|u- u_\eps\|_H=0$.
\end{lemma}
\begin{proof} 
Take $ u \in H_0$. 
By the triangle inequality, it holds
\begin{equation}\label{aux281}
	\|u- u_\eps\|_H\le \|(1-{\theta}_\eps)u\|_{\cQ} +  \|(1-{\theta}_\eps)\nabla u\|_{\cQ} + \|(\nabla{\theta}_\eps)u\|_{\cQ}.
\end{equation}
We use the notation $\cQ_\eps:=\text{supp}(1-\theta_\eps)$. 
For the first term in \eqref{aux281}, we have $\|(1-{\theta}_\eps)u\|_{\cQ}\le
\|u\|_{{\cQ}_\eps}\to 0 $ for $\eps\to 0$, since $u^2\in L^1(\cQ)$ if $u\in H$ and $\mbox{meas}({\cQ}_\eps)\to 0$, cf. \eqref{distance}. By the same argument the second term on the right hand side of \eqref{aux281} goes to zero with vanishing  $\eps$. For the third term we note $\|(\nabla{\theta}_\eps)u\|_{\cQ} = \|(\nabla\theta_\eps)u\|_{\cQ_\eps}$. In the rest of the proof we show that $\|(\nabla\theta_\eps)u\|_{\cQ_\eps}\to0$ for $\eps\to 0$. 
\smallskip

\noindent\textbf{A finite cover.} We introduce a finite cover of $\Qeps$ with cubes in $\R^{d+1}$ with side length $2 \delta$. The value of $\delta$ will be chosen sufficiently small, but independent of $\eps$. %, and is determined in the analysis below. 
This cover is denoted by $\{\cO_i\}_{0 \leq i \leq N}$, $N=N(\delta)$, $\cO_i=\{\, y \in \R^{d+1}~|~\|y-a^{(i)}\|_\infty < \delta\,\}$, with centers $a^{(0)}:=(0,0)$,  $a^{(i)} \in \Gamma_\cQ$, $1 \leq i\leq N$, and $\cQ_\eps \subset \cup_{0 \leq i \leq N} \cO_i$. We choose this cover such that $(0,0) \notin \overline{\cup_{1\leq i \leq N} \cO_i} $ and $|\nabla \phi| > 0$ on $\overline{\cup_{1\leq i \leq N} \cO_i}$. Note that
\begin{equation} \label{sumeq}
 \|(\nabla\theta_\eps)u\|_{\cQ_\eps}^2\leq \|(\nabla\theta_\eps)u\|_{\cO_0\cap \cQ_\eps}^2+  \sum_{i=1}^N\|(\nabla\theta_\eps)u\|_{\cO_i\cap \cQ_\eps}^2.
\end{equation}
We can make use of the density of smooth functions (Theorem~\ref{densityH0}) as follows. Consider a term $\|(\nabla\theta_\eps)u\|_{\cO_i\cap \cQ_\eps}$, $0  \leq i \leq N$.  Assume that there is a constant $C$, independent of $\eps$,  such that
\begin{equation} \label{mainestimate}
\|(\nabla\theta_\eps)u\|_{\cO_i\cap\cQ_\eps} \leq C \|\nabla u\|_{U_\eps} \quad \text{for all}~u \in C_c^\infty (\cQ)
\end{equation}
holds for a domain $U_\eps$ with $(\cO_i \cap \cQ_\eps) \subset U_\eps \subset \cQ$ with ${\rm meas}(U_\eps) \to 0$ for $\eps \to 0$. For the given $u \in H_0$ we take a sequence $(u_k)_{k \geq 1} \subset C_c^\infty(\cQ)$ with $\lim_{k\to \infty} u_k=u$ in $H$ and note:
\begin{align*}
  \|(\nabla\theta_\eps)u\|_{\cO_i\cap \cQ_\eps}  & \leq \|\nabla\theta_\eps(u-u_k)\|_{\cO_i\cap\cQ_\eps} + \|(\nabla\theta_\eps)u_k\|_{\cO_i\cap\cQ_\eps} \leq c \eps^{-1}\|u-u_k\|_H+ C \|\nabla u_k\|_{U_\eps}  \\
   & \leq \tilde c \eps^{-1}\|u-u_k\|_H + C \|\nabla u\|_{U_\eps}.
\end{align*}
Letting $k\to \infty$ and using $ \lim\limits_{\eps \to 0} {\rm meas}(U_\eps)=0$ we conclude $\lim\limits_{\eps \to 0} \|(\nabla\theta_\eps)u\|_{\cO_i\cap \cQ_\eps}=0$. Hence, it suffices to derive the estimate \eqref{mainestimate} for $u \in C_c^\infty (\cQ)$. 
\\[1ex]
\noindent\textbf{Estimate \eqref{mainestimate} away from $(0,0)$.} We first consider the ``regular'' terms $\|(\nabla\theta_\eps)u\|_{\cO_i\cap \cQ_\eps}^2$, $i \geq 1$ and with $u\in C_c^\infty(\cQ)$.  From $\nabla \theta_\eps= -h_\eps'\nabla \phi$ we get 
\[
  \|(\nabla\theta_\eps)u\|_{\cO_i\cap \cQ_\eps} \leq c \eps^{-1}\|u\|_{\cO_i\cap \cQ_\eps}.
\]
In the neighborhood $\cO_i$, $i \geq 1$,  with coordinate $y=(y_1, \ldots, y_{d+1})=(x,t)$, there is a coordinate $y_j$, with $j \leq d$, say $y_1$, such that $\left|\frac{\partial \phi}{\partial y_1}\right| > 0$ on $\overline {\cO_i}$. From the implicit function theorem it follows that there is a smooth function $g: \R^d \to \R$ such that, for $\delta$ sufficiently small, the local space-time boundary $\cO_i \cap \Gamma_\cQ$ can be represented as the graph of $g:\, U_\delta \to \R$, $U_\delta:=\{\, y \in \cO_i~|~y_1=0\,\}$. Furthermore, for a suitable constant $\hat c$ and $\eps$ sufficiently small we have $\cO_i \cap \cQ_\eps\subset \{\, (\hat y, g(\hat y)- \alpha)~|~\hat y \in U_\delta, ~0 \leq \alpha \leq \hat c \eps\,\}$. 
%Thanks to Theorem~\ref{densityH}, for a given $u \in H_0$ we can define $\{u_k\}_{k \geq 1} \subset C^1(\overline \cQ)$ such that 
%$$\lim_{k\to \infty} \|u_k-u\|_H=0.$$ 
%We note that
%\begin{equation} \label{A1}
%  \|u_k\|_{\cO_i \cap \cQ_\eps}^2 \leq \int_{U_\delta} \int_0^{\hat c \eps} u_k(\hat y, \alpha)^2 \, d\alpha \, d\hat y
%\end{equation}
%holds and apply a standard 1D Poincare estimate to obtain
%For $\hat y \in U_\delta$ we have $u_k(\hat y, \alpha)= u_k(\hat y,0)+ \int_0^\alpha \frac{\partial u_k}{\partial y_1}(\hat y, \xi) \, d\xi$. This yields, for $\alpha \leq \hat c\eps$,
%\[
%  u_k(\hat y,\alpha)^2 \leq 2 u_k(\hat y, 0)^2 + 2 \alpha \int_0^{\hat c \eps} |\nabla u_k(\hat y,\xi)|^2\, d\xi.
%\]
%Integration over $U_\delta \times [0,\hat c \eps]$ and using the boundedness of $|\nabla g|$ on $U_\delta$ we obtain
%\begin{equation} \label{A2}
%  \int_{U_\delta} \int_0^{\hat c \eps} u_k^2 \, d\alpha \, d\hat y \leq c \, \eps \|u_k\|_{ \Gamma_\cQ \cap \cO_i}^2 %+ \hat c^2\eps^2 \int_{U_\delta} \int_0^{\hat c \eps} |\nabla u_k|^2 \, d\alpha \, d\hat y .
%\end{equation}
%Thanks to  $|V_\Gamma| \leq c$ on $\cO_i \cap \cQ_\eps$, $i \geq 1$, the boundary term above is bounded as  $\|u_k\|%_{ \Gamma_\cQ \cap \cO_i}\le \|u_k\|_{L^2_\omega}(\GQ)$. 
%When passing to the limit $k\to \infty$  the boundary term vanishes, since 
%\begin{equation}\label{bound}
%	\|u_k\|_{L_\omega^2(\GQ)}=\|u_k-u\|_{L_\omega^2(\GQ)}\le c \|u_k-u\|_{H}\to 0.
%\end{equation}
We apply a standard 1D Poincare estimate to obtain
\begin{equation} \label{Aux1}
 \eps^{-2}\|u\|_{\cO_i\cap \cQ_\eps}^2 \leq \eps^{-2} \int_{U_\delta} \int_0^{\hat c \eps} u^2 \, d\alpha d \hat y \leq c\int_{U_\delta} \int_0^{\hat c \eps} |\nabla u|^2 \, d\alpha d\hat y.
\end{equation}
This proves \eqref{mainestimate} with $U_\eps=U_\delta \times [0,\hat c \eps]$, which satisfies  ${\rm meas}(U_\eps)\to 0$ for $\eps \to 0$. \\[1ex]

\noindent\textbf{Estimate \eqref{mainestimate} in a neighborhood of $(0,0)$.}
Consider now the critical term 
$\|(\nabla \theta_\varepsilon) u\|_{\cO_0 \cap \cQ_\varepsilon}$. 
Let $\delta>0$ be sufficiently small such that $\cO_0$ is contained in the 
neighborhood $\hat X$ defined in Lemma~\ref{L:Morse}. 
The map $\Psi$ from \eqref{transform2} induces the coordinate transformation 
$\Psi : (x,t) \mapsto (\psi(x,v^{-1}(t)), v^{-1}(t)) =: (z,t)$ on $\cO_0$.

Define $\widehat\phi(z,t) := \phi \circ \Psi$ and 
$\widehat{\theta}_\varepsilon(z,t) := \theta_\varepsilon \circ \Psi$. 
Note that $\widehat\phi$ has the normal form given in \eqref{normalform2}. 
Since $\Psi$ is $C^\infty$-smooth, transforming the integral 
$\|(\nabla \theta_\varepsilon) u\|_{\cO_0 \cap \cQ_\varepsilon}^2$ 
from $(x,t)$-coordinates to $(z,t)$-coordinates introduces only finite multiplicative constants.
The domain of integration 
$\cO_0 \cap \cQ_\varepsilon = \{ (x,t)\in \cO_0 \mid \theta_\varepsilon(x,t) < 1 \}$ 
is mapped to $\Psi(\cO_0 \cap \cQ_\varepsilon)$. 
One easily verifies that 
$\Psi(\cO_0 \cap \cQ_\varepsilon) 
= \{ (z,t)\in \Psi(\cO_0) \mid \widehat{\theta}_\varepsilon(z,t) < 1 \}$. 
Hence, without loss of generality, we may assume that $\phi$ is in the normal form specified in \eqref{normalform2}.

Using the normal form \eqref{normalform2}, we next estimate 
$\|(\nabla \theta_\varepsilon) u\|_{\cO_0 \cap \cQ_\varepsilon}$ 
and distinguish the cases of a degenerating island, domain merging or splitting, and hole formation.

\smallskip
\noindent\textbf{\small Degenerating island.} We consider  the creation of an island scenario. The case of a vanishing island is then also covered by replacing $t$ with $-t$. Thus $\phi(x,t)=|x|^2 -t$ in $\cO_0$. Recall that  $\cO_0$ has center $(0,0)$ and  $t \in [0,\delta]$. %and assume $\delta$ is sufficiently small such that $v'(t) > 0$ for all $t \in [-\delta,\delta]$. 
Note that
\[
  \cO_0\cap \cQ_\eps \subset 
  \{\, (x,t) \in \cO_0\,|\,-3\eps+ t \leq |x|^2 \leq t\,\}=:\cQ_\eps^{c}.
\]
The domain $\cQ_\eps^{c}$ does not contain $(x,t)$ with $t<0$. %The spatial domain is denoted by $\Omega(t)=\{\, x\in \R^d\,|\, (x,t) \in \cQ_\eps^{c}\,\}$.  
Since $\nabla \theta_\eps = h_\eps'\nabla \phi$ we obtain
\begin{equation} \label{E1}
   \|(\nabla\theta_\eps)u\|_{\cO_0\cap \cQ_\eps}^2   \leq \|(\nabla\theta_\eps)u\|_{\cQ_\eps^{c}}^2 \leq \eps^{-2}\int_{\cQ_\eps^{c}} |\nabla \phi|^2 u^2 \, dq = \eps^{-2}\int_{\cQ_\eps^{c}} |x|^2 u^2\, dq.                                                                                                                                                                                                                                                                                                                                                                                         
\end{equation}
We use radial coordinates and thus obtain, with $[z]_+=\max\{z,0\}$: 
\begin{equation} \label{B1}
 \eps^{-2}\int_{\cQ_\eps^{c}} |x|^2 u^2\, dq  \leq  \eps^{-2}\int_0^\delta \int_{S^{d-1}} \int_{[t-3 \eps]_+^\frac12}^{\sqrt{t}} r^2 u^2 r^{d-1} \, dr \, ds\, dt.
\end{equation}
We study the inner integral. Note that $u=0$ on $\Gamma_\cQ$ and thus  $u|_{r=\sqrt{t}}=0$ allowing us to apply the Hardy estimate \eqref{ResP1}. We first rewrite: 
\begin{equation} \label{B2}
 \int_{[t-3 \eps]_+^\frac12}^{\sqrt{t}} r^2 u^2 r^{d-1} \, dr= \int_{[t-3 \eps]_+^\frac12}^{\sqrt{t}} \big(r (\sqrt{t}-r)\big)^2 \left(\frac{u}{\sqrt{t}-r}\right)^2 r^{d-1} \, dr.
\end{equation}
For the term $\big(r (\sqrt{t}-r)\big)^2$  straightforward computations give: For $0 \leq t \leq 3 \eps$ and $0 \leq r \leq \sqrt{t}$ it holds that  $\big(r (\sqrt{t}-r)\big)^2 \leq r^2 t \leq t^2 \leq 9 \eps^2$. For $t \geq 3\eps$ and $\sqrt{t-3 \eps} \leq r \leq \sqrt{t}$ we have
\[
 \big(r (\sqrt{t}-r)\big)^2 \leq r^2 \big(\sqrt{t}-\sqrt{t-3\eps})^2 \leq \frac{9 \eps^2 t}{(\sqrt{t}+\sqrt{t-3\eps})^2} \leq 9 \eps^2.
\]
Using this in \eqref{B2}, applying the Hardy inequality \eqref{ResP1} and going back to \eqref{B1} we obtain, 
\[ \eps^{-2}\int_{\cQ_\eps^{c}} |x|^2 u^2\, dq  \leq 9  C_{d-1}\int_0^\delta \int_{S^{d-1}} \int_{[t-3 \eps]_+^\frac12}^{\sqrt{t}} \left(\frac{\partial u}{\partial r}\right)^2  r^{d-1} \, dr \, ds\, dt \leq  9  C_{d-1} \|\nabla u \|_{\cQ_\eps^{c}}^2.
\]
Using this in \eqref{E1} and noting that ${\rm meas}(\cQ_\eps^{c}) \to 0$ for $\eps \to 0$ proves \eqref{mainestimate}  for $i=0$. 
\smallskip

\noindent\textbf{\small Domains merging and splitting.}
We now turn to the case of domain merging or splitting. We consider the 2D case, i.e.,  $\phi(x,t)=-x_1^2+x_2^2+t$  (domain splitting, while domain merging is covered by changing $t\to-t$). 
%We assume $\delta $ is sufficiently small such that $v'(t) \geq v_0 >0$ for all $t \in [-\delta,\delta]$. This implies $v_0|t| \leq |v(t)|$ for suitable $v_0>0$ for all $t \in [-\delta,\delta]$. 
As above we have $\cO_0 \cap \cQ_\eps \subset \{\, (x,t) \in \cO_0\,|\,-3 \eps \leq \phi(x,t)\leq 0\,\}=:\cQ_\eps^{c}$.
\\
We first consider $t \in [0,\delta]$ and note that $-3 \eps \leq \phi(x,t)\leq 0$ iff $x_2^2 +t \leq x_1^2 \leq x_2^2 +t+3 \eps$. We parameterize the domain $\cQ_\eps^{c} \cap (t \geq 0)$ over $x_2 \in [-\delta,\delta]$ and due to symmetry we have
\begin{equation} \begin{split} \label{E4}
 \|(\nabla\theta_\eps)u\|_{\cQ_\eps^{c} \cap (t \geq 0)}^2  & \leq c \eps^{-2} \||x| u\|_{\cQ_\eps^{c} \cap (t \geq 0)}^2  \\ & = 2 c \eps^{-2}\int_0^\delta \int_{-\delta}^\delta \int_{\sqrt{x_2^2 +t}}^{\sqrt{x_2^2 +t+3 \eps}} u^2 |x|^2\, dx_1 \, dx_2 \, dt. 
\end{split} \end{equation}
 Rewriting  the inner integral yields
\begin{equation} \label{E5} \begin{split}
  & \eps^{-2}\int_{\sqrt{x_2^2 +t}}^{\sqrt{x_2^2 +t+3 \eps}} u^2 |x|^2\, dx_1  \\ & = \eps^{-2}\int_{\sqrt{x_2^2 +t}}^{\sqrt{x_2^2 +t+3 \eps}} \left(\frac{u}{x_1-\sqrt{x_2^2 +t}}\right)^2 \big(x_1-\sqrt{x_2^2 +t}\big)^2|x|^2\, dx_1.
\end{split}\end{equation}
For $\sqrt{x_2^2 +t} \leq x_1 \leq \sqrt{x_2^2 +t+3 \eps}$ we have $|x_1-\sqrt{x_2^2 +t}| \leq \sqrt{x_2^2 +t+3 \eps}- \sqrt{x_2^2 +t}$ and so 
\[
  \big(x_1-\sqrt{x_2^2 +t}\big)^2|x|^2 \leq 9 \eps^2 \frac{x_1^2 +x_2^2}{\big(\sqrt{x_2^2 +t+3 \eps}+\sqrt{x_2^2 +t}\big)^2} \leq 9 \eps^2.
\]
 Noting  that $u|_{x_1=\sqrt{x_2^2+t}}=0$ and using this in \eqref{E5} together with the  Hardy inequality \eqref{ResP1} we obtain
\begin{equation} \label{HH1}
  \eps^{-2}\int_{\sqrt{x_2^2 +t}}^{\sqrt{x_2^2 +t+3 \eps}} u^2 |x|^2\, dx_1 \leq 9 C_0 \int_{\sqrt{x_2^2 +t}}^{\sqrt{x_2^2 +t+3 \eps}} \left(\frac{\partial u}{\partial x_1}\right)^2 \, dx_1.
\end{equation}
Substituting this in \eqref{E4} yields
\begin{equation} \label{E6}
  \|(\nabla\theta_\eps)u\|_{\cQ_\eps^{c} \cap (t \geq 0)}^2 \leq c \| \nabla u\|_{\cQ_\eps^{c} \cap (t \geq 0)}^2,
\end{equation}
with a suitable constant $c$ independent of $\eps$. For $t \in [-\delta,0]$ we use very similar arguments, but parameterize over $x_1$, noting that $-3 \eps \leq \phi(x,t)\leq 0$ iff $x_1^2 -t-3 \eps \leq x_2^2 \leq x_1^2 -t $, we obtain
\begin{equation} \label{HH2}
 \|(\nabla\theta_\eps)u\|_{\cQ_\eps^{c} \cap (t \leq 0)}^2  \leq c \eps^{-2}\int_{-\delta}^0 \int_{-\delta}^\delta \int_{[x_1^2 -t-3\eps]_+^\frac12}^{\sqrt{x_1^2 -t}} u^2 |x|^2\, dx_2 \, dx_1 \, dt.
\end{equation}
For the inner integral we have 
\begin{equation} \label{HH3}  \eps^{-2}\int_{[x_1^2 -t-3\eps]_+^\frac12}^{\sqrt{x_1^2 -t}} u^2 |x|^2\, dx_2
 = \eps^{-2}\int_{[x_1^2 -t-3\eps]_+^\frac12}^{\sqrt{x_1^2 -t}} \left| \frac{u}{\sqrt{x_1^2 -t} -x_2}\right|^2 \big(\sqrt{x_1^2 -t} -x_2\big)^2 |x|^2\, dx_2,
\end{equation}
and for $[x_1^2 -t-3\eps]_+^\frac12 \leq x_2 \leq\sqrt{x_1^2 -t}$ we have
\begin{equation} \label{HH4}
  \big(\sqrt{x_1^2 -t} -x_2\big)^2 |x|^2 \leq 9 \eps^2 \frac{x_1^2+x_2^2}{\big(\sqrt{x_1^2 -t}+[x_1^2 -t-3\eps]_+^\frac12\big)^2} \leq 18\eps^2.
\end{equation}
By the same arguments as above we obtain $\|(\nabla\theta_\eps)u\|_{\cQ_\eps^{c} \cap (t \leq 0)}^2 \leq c \| \nabla u\|_{\cQ_\eps^{c} \cap (t \leq 0)}^2$, with $c$ independent of $\eps$. Combining this with the result \eqref{E6} we conclude that \eqref{mainestimate} holds for $i=0$. The case  of domain merging or splitting for $d=3$ can be handled in a very similar way, cf. Appendix~\ref{AppB}.

\smallskip
\noindent\textbf{\small Hole through the domain.}
The case  of creation or vanishing of a hole through the domain (case 3c)) can be treated similarly.  Details are given in Appendix~\ref{AppB}.\quad
\end{proof}
\smallskip

\begin{remark} \label{RemdiscussionHoles} \rm   We explain why the cases 2c) and 3d) are excluded from the statement of Lemma~\ref{lemA}. 
The triangle inequality \eqref{aux281} implies that in order to have $\lim\limits_{\eps \to 0}\|u- u_\eps\|_H=0$  the term $\|(\nabla{\theta}_\eps)u\|_{\cQ_\eps}$ should vanish for $u \in H_0$ and $\eps \to 0$. The estimate \eqref{mainestimate} for $i=0$ is critical in this regard.
Let $d=2$ and consider the creation of the hole scenario, case 2c). Similar arguments apply for $d=3$. We use   polar coordinates $(r, \psi)$ around the critical point $x=0$. Consider the evolving domain $\Omega(t)=\{\, x \in \R^2\,|\, t \leq r^2 \leq 2\,\}$, $t \in [-1,1]$ with $t_c=0$. 
%by  Two $\eps$-narrow strips of $\Omega(t)$ in the proximity of the critical point are given in  polar coordinates $x=r (\cos \psi, \sin \psi)$ by 
The spatial slices of ${\rm supp}(|\nabla \theta_\eps|)$, cf. \eqref{distance}, are contained in a strip $S(t) \subset \Omega(t)$:
	\begin{align*}
%\Sin(t) = \{\, x~|~ t\leq r^2 \leq t+3\eps \, \} \quad \text{and}\quad 
  {\rm supp}\big(\|\nabla \theta_\eps(\cdot, t)|\big)\subset S(t) := \{\, x~|~ t+\eps\leq r^2 \leq t+3\eps \, \} .   
\end{align*}
 Note that this support has  a uniform width $\sim \eps$  is space-time, but that the width of the spatial slices depends on $\eps$ and $t$. 
Without loss of generality we may assume that $\cO_0$ is a space--time box. Then for  the term at the left-hand side of \eqref{mainestimate} we  have 
\[
 \|(\nabla{\theta}_\eps)u\|_{\cO_0 \cap \cQ_\eps}^2= \int_{-\delta}^\delta\|(\nabla{\theta}_\eps)u\|_{S(t)}^2\,dt 
\]
for some  $\delta>0$. 
In $S(t) $ it holds that $|\nabla \theta_{\eps}|=|\nabla h_{\eps}(-\phi)| \sim \frac{r}{\eps}$ and so 
\begin{equation} \label{A4}
\|(\nabla{\theta}_\eps)u\|_{\cO_0 \cap \cQ_\eps}^2\approx
\int_{-\delta}^\delta\left\|\frac{r}{\eps}u\right\|_{L^2(S(t))}^2\,dt.
\end{equation}
In the proof of Lemma~\ref{lemA} the key result \eqref{mainestimate} in a neighborhood of the critical point is derived  (essentially)
based on  
	\begin{equation}\label{aux652}
\left\|\frac{r}{\eps}u\right\|_{L^2(S(t))} \le C \|u\|_{H^1(S_{\rm ex}(t))},
	\end{equation}
with $S_{\rm ex}(t):=	\{\, x~|~ t\leq r^2 \leq t+3\eps \, \}$ a strip that is (slightly) larger than $S(t)$ and with a constant $C$ independent of $\eps$ and $t$. 
The estimate \eqref{aux652}, however, fails for the case 2c)  as can be seen from the following. 
%For $t \in (-3 \eps, -\eps)$ the  strip $S(t)$ is nonempty and the domain $\Omega(t)$ does not contain a hole, yet. Then we can chose $u \in H_0(\Omega(t))$ that is constant on $S_{\rm ex}(t)$ and for such a $u$ the estimate \eqref{aux652} does not hold. But also if we consider the case that $t_0=t_c=0$, i.e., the hole is created at the initial point in time, estimate \eqref{aux652} can not be correct, as the following example shows. 
Take $t \in (0,1]$ and $u(x)=\ln \left( \frac{r}{\sqrt{t}}\right)$, $r \geq \sqrt{t}$. Note that $u \in H^1(S_{\rm ex}(t))$ and satisfies $u(x)=0$ if $|x|=\sqrt{t}$. Restrict to $\frac{t}{\eps} \ll 1$ and introduce the notation $\xi:=\eps/t$. straightforward computations yield:
\begin{align*}
	\left\|\frac{r}{\eps}u\right\|_{L^2(S(t))}^2 & = \int_0^{2 \pi} \xi^{-2} \int_{\sqrt{1+\xi}}^{\sqrt{1+3\xi}} y^3 \ln^2y \, dy \, d \psi  \sim \xi^{-2} \big( \xi^2 \ln^2 \xi) = \ln^2 \xi \quad \text{for}~ \xi \to \infty, \\
	\left\|\nabla u\right\|_{L^2(S_{\rm ex}(t))}^2 & = \int_0^{2 \pi} \int_{\sqrt{t}}^{\sqrt{t+3\eps}} \left(\frac{\partial u}{\partial r}\right)^2  r \, dr \, d\psi = \int_0^{2 \pi}  \int_{\sqrt{1}}^{\sqrt{1+3\xi}} \frac{1}{y} \, d y \, d \psi \sim\ln \xi \quad \text{for}~ \xi \to \infty.
\end{align*}
It also holds $\|u\|_{H^1(S_{\rm ex}(t))}\le C \|\nabla u\|_{L^2(S_{\rm ex}(t))}$ with $C$ independent of $\xi$.
Hence, a uniform (in $\eps$ and $t$) estimate as in \eqref{aux652} cannot hold. 
\end{remark}
\smallskip
 \rev{\begin{corollary} \label{corolnew} Under  the same assumptions as in Lemma~\ref{lemA} we have the stability estimate
 \begin{equation} \label{stabestimate}
   \|\theta_\eps u\|_H \leq c \|u\|_H \quad \text{for all}~ u \in H_0,
 \end{equation}
with a constant $c$ independent of $u$ and $\eps$.
 \end{corollary}  
 \begin{proof}
 In the proof of Lemma~\ref{lemA} we derived the estimate \eqref{mainestimate}, and combining this with \eqref{sumeq} yields  $\|(\nabla \theta_\eps) u\|_{\cQ} \leq c \|\nabla u \|_{\cQ}$ for all $u \in C_c^\infty(\cQ)$. Hence,
 \[
   \|\theta_\eps u \|_H \leq  \|\theta_\eps u \|_\cQ + \|(\nabla \theta_\eps) u\|_{\cQ} +\|\theta_\eps \nabla u\|_{\cQ}\leq c \|u\|_H \quad \text{for all}~u \in C_c^\infty(\cQ)
 \]
holds with a suitable constant $c$ independent of $u$ and $\eps$. Thanks to Theorem~\ref{densityH0} and
$\theta_\eps\in C^\infty(\overline\cQ)$ we conclude \eqref{stabestimate} by a standard density argument.
% From $\nabla \theta_\eps= h_\eps' \nabla \phi$ we obtain $\|(\nabla \theta_\eps) w\|_{\cQ} \leq \hat c \eps^{-1} \|w\|_{H}$ for all $w \in H_0$ with $\hat c$ independent of $w$ and $\eps$. Let $u \in H_0$ be given. For $\eps >0$, take $u_\eps \in C_c^\infty(\cQ)$ such that $\|u-u_\eps\|_H \leq \eps \|u\|_H$ holds. We then have
%\[
% \|\theta_\eps u\|_H \leq \|\theta_\eps(u-u_\eps)\|_H + \|\theta_\eps u_\eps\|_H \leq 
%  \hat c \eps^{-1} \|u-u_\eps\|_H+ c \|u_\eps\|_H \leq \big (\hat c+c(1+\eps)\big) \|u\|_H \leq c \|u\|_H,
%\]
%with a suitable constant $c$ independent of $\eps$ and $u$.
 \end{proof}
 }
\begin{lemma} \label{CorA} Under the same assumptions as in Lemma~\ref{lemA}, we have $\lim\limits_{\eps \to 0} \frac{\|u\|_{\cQ_\eps}^2}{\eps}=0$ for all $u \in H_0$.
\end{lemma}
\begin{proof}
We use the notation introduced in the proof of Lemma~\ref{lemA}. 
By the same density argument as used below \eqref{mainestimate}, it suffices to prove
\begin{equation} \label{suffi}
	\varepsilon^{-1}\|u\|_{\cO_i \cap \cQ_\varepsilon}^2 
	\le C \|\nabla u\|_{U_\varepsilon}^2 
	\qquad \text{for all } u \in C_c^\infty(\cQ).
\end{equation}

For the regular sets $\cO_i$, $1 \le i \le N$, this follows from \eqref{Aux1}. 
For the neighborhood $\cO_0$ of the critical point, we distinguish the same cases as in the proof of Lemma~\ref{lemA}. 

First consider the degenerating island. We need to estimate the integral in \eqref{E1} without the factor $|x|^2$. 
This yields \eqref{B2} without the factor $r^2$. 
One checks that for $r \in \big[ \rev{[t-3\varepsilon]_+^{1/2}}, \sqrt{t} \big]$ the inequality 
$(\sqrt{t}-r)^2 \le 3\varepsilon$ holds. 
Proceeding as in the proof of Lemma~\ref{lemA}, we obtain \eqref{suffi}. 

Next consider the domain splitting case with $t \ge 0$, which leads to \eqref{E4} without the factor $|x|^2$. 
As in \eqref{E5}, one verifies that for 
$x_1 \in \big[\sqrt{x_2^2+t}, \sqrt{x_2^2+t+3\varepsilon}\big]$ the estimate 
$\left(x_1-\sqrt{x_2^2+t}\right)^2 \le \varepsilon$ holds. 
Following the same arguments as in the proof of Lemma~\ref{lemA}, we again obtain \eqref{suffi}. 
The case $t \le 0$, cf.~\eqref{HH3}--\eqref{HH4}, can be treated analogously. 
The same arguments also apply to the cases of domain merging/splitting for $d=3$ and to the hole-through-the-domain case.\quad 
\end{proof}
\smallskip

We are now prepared to study  properties of the space $W$.  We start by proving the density of smooth functions. 
To this end, we first show that functions with spatially compact support are dense in $W$.

\begin{lemma}\label{L2} 
Consider a space-time domain with a critical point as in  classification \ref{Class}, except for the cases 2c) and 3d).
For $u\in W$, there exists a sequence $(w_n)_{n \geq 1} \subset W$ with $\text{\rm dist}\big(\text{\rm supp}(w_n),\Gamma_\cQ\big) > 0$ for all $n$ and  $ w_n\to u$ in $W$.  
\end{lemma}     
\begin{proof} Take $u \in W$ and consider the cut off with $\theta_\eps$ defined in \eqref{cutoff}.  We note that $\theta_\eps u\in W$, since $\theta_\eps\in C^\infty(\overline\cQ)$. Let $\eps_n=\eps_0 2^{-n}$, $n \in \mathbb{N}$, with  $\eps_0 >0$ sufficiently small.  Define $u_n:=\theta_{\eps_n} u$.  
Lemma~\ref{lemA} implies  $u_n\to u$ in $H$.
	By  definition, we have for any $\xi\in C^\infty_c(\cQ)$: 
	\begin{equation}\label{aux395}\begin{split}
			\langle \partial_t u_n,\xi\rangle & = -\int_{\cQ} \theta_{\eps_n} u\,\partial_t \xi\, dq =
			-\int_{\cQ}  u\,\partial_t (\xi\theta_{\eps_n})\, dq + \int_{\cQ}  u\, \xi\, \partial_t \theta_{\eps_n}\, dq  \\ & =\langle \partial_t u,\theta_{\eps_n} \xi\rangle+ \int_{\cQ}  u\, \xi\, \partial_t \theta_{\eps_n}\, dq.
	\end{split} \end{equation}
\rev{Using Corollary~\ref{corolnew} we obtain for the first term on the right hand side of \eqref{aux395}:
	\[ |\langle \partial_t u,\theta_{\eps_n} \xi\rangle| \leq c \|u\|_W \|\xi\|_H,   
	\]
with a constant $c$ independent of $n$.}
For  the second term, 
  using  Lemma~\ref{CorA} we obtain
	\[
	\begin{split}
	\frac{\left| \int_{\cQ}  u\, \xi\, \partial_t \theta_{\eps_n}\, dq \right|}{\|\xi\|_H}& = 
	\frac{\left| \int_{\cQ}  u\, \xi\, h_{\eps_n}'(-\phi) \partial_t \phi \,dq \right|}{\|\xi\|_H}
	\le c\,\eps_n^{-1} \frac{\int_{\cQ_{\eps_n}}  |u\, \xi|\,dq}{\|\xi\|_H} \le \eps_n^{-\frac12}  \|u\|_{\cQ_{\eps_n}} \frac{ \eps_n^{-\frac12} \|\xi\|_{\cQ_{\eps_n}}}{\|\xi\|_H} \\ 
	&   \rev{ \le c\, \eps_n^{-\frac12}  \|u\|_{\cQ_{\eps_n}}  }  \to 0 ~\text{for}~n \to \infty. 
	\end{split}
	\]
\rev{Where for $ \eps_n^{-\frac12} \|\xi\|_{\cQ_{\eps_n}}\le C\|\xi\|_H$ we use the result in \eqref{suffi} after summing over the cover. Thus we conclude
	\begin{equation}\label{UnifBound}
\sup_{n \in \mathbb{N}} \sup_{0\neq \xi\in C^\infty_c(\cQ)} \frac{|\langle \partial_t u_n,\xi\rangle|}{\|\xi\|_H} \le C,
	\end{equation}
}
and since $\|\xi- \theta_{\eps_n} \xi\|_H\to 0$ for $n\to \infty$, eq.~\eqref{aux395} also implies
\[
	\lim_{n\to \infty}\langle \partial_t u_n - \partial_t u ,\xi\rangle =0  \quad \text{for all}~\xi\in C^\infty_c(\cQ).
\]
	Using the density of $C^\infty_c(\cQ)$ in $H_0$, \rev{continuity of the functional $\partial_t u_n$, and the uniform boundedness in \eqref{UnifBound}
	we conclude that the sequence $\partial_t u_n$ is uniformly bounded in $H^{-1}_0$. }
Therefore, we obtain the	weak convergence $ \partial_t u_n \rightharpoonup \partial_t u$ 
	in $H_0^{-1}$. By the Banach--Saks theorem there exists a subsequence of $(u_n)_{n \geq 1}$, which we also denote by $(u_n)_{n \geq 1}$, such that the Cesaro means of $\{(\partial_tu_n)\}_{n \geq 1}$ converge strongly in $H_0^{-1}$, i.e,
	\[
	\lim_{k \to \infty} \frac{1}{k} \sum_{n=1}^k\partial_t u_n =\partial_t u  \quad \text{strongly in}~ H_0^{-1}.
	\]
	Define $w_k:=\frac{1}{k} \sum_{n=1}^k u_n $. From $u_n\to u$ in $H$ it follows that $w_k\to u$ in $H$. We also have $\partial_t w_k\to \partial_t u$ in $H_0^{-1}$. Hence $w_k\to u$ in $W$. From $\text{\rm dist}\big(\text{\rm supp}(u_n),\Gamma_\cQ\big) > 0$ for all $n$ it follows that $\text{\rm dist}\big(\text{\rm supp}(w_k),\Gamma_\cQ\big) > 0$ for all finite $k$.\\
\end{proof}

\begin{lemma}\label{L3} Consider a space-time domain with critical points as in  classification \ref{Class}, except for the cases 2c) and 3d). The subspace of functions from $C^\infty(\overline \cQ)$ vanishing on $\Gamma_\cQ$ is  dense in $W$. 
\end{lemma}     
\begin{proof} %Here we may follow the argument from~\cite{lions1957problemes}.
 From Lemma~\ref{L2} it follows that it is sufficient to consider  $u\in W$ for which $\epsilon:=\text{dist}(\text{supp}(u),\Gamma_\cQ) > 0$ holds and to show that $u$ can be approximated arbitrary well by smooth functions in $\|\cdot\|_W$. 
 Since $\Omega(t)$ is bounded for all $t\in [-1,1]$, there is a ball $\widehat B\in\R^d$ and a space--time cylinder $\widehat\cQ= \widehat B\times[-1,1]\in\R^{d+1}$ such that $\cQ\subset \widehat\cQ$. For functions defined on $\widehat \cQ$ we introduce the standard Bochner space
\[
\widehat W =\{u\in L^2([-1,1],H^1_0(\widehat B))\,|\,\partial_t u \in L^2([-1,1],H^{-1}_0(\widehat B))\},
\]
with the corresponding norm. 
For $u\in W$ denote by $\bar u$ its extension by zero to   $\widehat\cQ$.
Since $u\in H_0$ it is straightforward to see that $\bar u\in L^2([-1,1],H^1_0(\widehat B))$
and
\[
\|\bar u\|_{ L^2([-1,1],H^1_0(\widehat B))} =\|u\|_H.
\]
Since $\eps =\text{dist}(\text{supp}(u),\Gamma_\cQ) >0$,  there is a cut-off function $\kappa_\eps= \kappa_\eps(u) \in C^\infty(\overline{\widehat\cQ})$, such that $\text{supp}(\kappa_\eps)\subset \overline \cQ$,  $\text{dist}(\text{supp}(\kappa_\eps),\Gamma_\cQ)\ge \eps/2$ and
$\kappa_\eps=1$ on $\text{supp}(u)$. 
For $\xi \in C_c^\infty(\widehat \cQ)$ we have
\begin{align*}
 \langle \partial_t \bar u, \xi \rangle & = - \int_{\widehat \cQ} \bar u \, \partial_t \xi \, dq =- \int_\cQ u\, \partial_t(\kappa_\eps \xi) \, dq 
 = \langle \partial_t  u, \kappa_\eps \xi \rangle.
\end{align*}
Using this, $\partial_t u \in H_0^{-1}$ and $\|\kappa_\eps \xi\|_H \leq c(\eps) \|\xi\|_{ L^2([-1,1],H^1_0(\widehat B))}$ it follows that $\bar u \in\widehat W$. One easily checks that $\|u\|_W \leq \|\bar u\|_{\widehat W}$ holds.  By a standard  result for the cylindrical space-time domain, we have that for any $\delta>0$ there exists $\bar u_\delta\in C^\infty(\overline {\widehat\cQ})$ %vanishing on  the lateral boundary of $\widehat \cQ$
 such that 
\begin{equation}\label{aux450}
\|\bar u- \bar u_\delta\|_{\widehat{W}}\le \delta.
\end{equation}
Define $ u_\delta= \kappa_\eps \bar u_\delta \in C^\infty(\overline{\widehat\cQ})$.  Note that $\text{supp}(u_\delta)\subset \overline \cQ$ holds and $u_\delta$ vanishes on $\Gamma_\cQ$. For $g \in C^\infty(\overline {\widehat\cQ})$ and  $w \in \widehat{W}$ we have $\|g w\|_{\widehat W} \leq c_g \|w\|_{\widehat W}$, for a suitable constant $c_g$ independent of $w$.  Using this and triangle inequalities we obtain, for  suitable $c_\eps=c_\eps(u) >0$,  
\[
\begin{split}
\|u-u_\delta\|_W&\le 
\|\bar u-u_\delta\|_{\widehat{W}}\le
\|\bar u-\bar u_\delta\|_{\widehat{W}} + \|\bar u_\delta-  u_\delta\|_{\widehat{W}} \le 
\delta + \|\bar u_\delta-  u_\delta\|_{\widehat{W}} \\ &=
\delta + \|(1-\kappa_\eps) \bar u_\delta\|_{\widehat{W}}
=
\delta + \|(1-\kappa_\eps)(\bar u-\bar u_\delta)\|_{\widehat{W}}
\le \delta + c_\eps\|\bar u-\bar u_\delta\|_{\widehat{W}}\\
&\le (1+c_\eps)\delta.
\end{split}
\]
For given $u$ and $\epsilon > 0$ we take $\delta> 0$ such that $(1+c_\eps(u))\delta < \epsilon$ and for the corresponding $\bar u_\delta$ as in \eqref{aux450} we take $u_\delta= \kappa_\eps \bar u_\delta  \in C^\infty(\overline{\cQ})$, which vanishes on $\Gamma_\cQ$ and satisfies $\|u-u_\delta\|_W \leq \epsilon$. This completes the proof.
\end{proof}

The next lemma establishes a Lions--Magenes-type result for $W$. 
%We state and prove it for the case of merging/splitting domains. For the case of a vanishing or emerging island, the lemma and its proof remain valid if the time interval $[-1,1]$ is replaced by $[-1,0)$ and $(0,1]$, respectively.

\begin{lemma}\label{L4} Consider a space-time domain with critical points as in  classification \ref{Class}, except for the cases 2c) and 3d). Let $u\in W$. 
Then for any $t\in[t_0,T]$ the trace of $u$ on $\Omega(t)$ is well defined as an element of $L^2(\Omega(t))$ and 
\begin{equation}\label{AMest}
\|u(t)\|_{\Omega(t)}\le C\, \|u\|_W,
\end{equation}
with some $C>0$ independent of $t$ and $u$.
\end{lemma}     
\begin{proof}  For $t_0 \leq t_1< t_2\leq T$, define $\cQ(t_1,t_2)= \bigcup\limits_{s \in (t_1,t_2)} \Omega(s) \times \{s\}$ and  
$H_0(t_1,t_2)=\{u\in L^2(\cQ(t_1,t_2))\,|\, \nabla u\in L^2(\cQ(t_1,t_2))^d,~~ u_{|\Gamma_{\cQ(t_1,t_2)}}=0\,\}$. By the same arguments as in the proof of Theorem~\ref{densityH0} the functions from $C^\infty_c(\cQ(t_1,t_2))$ are dense in $H_0(t_1,t_2)$.
Define the corresponding space $W(t_1,t_2)$ as in \eqref{defW}.  
%It is therefore holds for any smooth $\bar u\in W$, $\bar v\in H$, 
One easily checks that $u \in W$ implies $u_{|\cQ(t_1,t_2)} \in W(t_1,t_2)$ and that (with the restriction also denoted by $u$)
\begin{equation}\label{aux517}
\|u\|_{W(t_1,t_2)} \leq \|u\|_W
%\int_{\cQ(t_1,t_2)}\bar u_t\bar v\, dq \le  \|\bar u_t\|_{H^{-1}_0(t_1,t_2)} \|\bar v\|_{H(t_1,t_2)}
% \le\|\bar u\|_W \|\bar v\|_{H}.
\end{equation}
holds. Consider the cases 2b), 3b) or 3c), i.e. $[t_0,T]=[-1,1]$ and introduce  a quasi-space--time cylinder $\cQ(-1,-\tfrac12)$ (for the case of emerging island, with $[t_0,T]=[0,1]$, one considers instead the cylinder $\cQ(\tfrac12,1)$ and makes other obvious modifications). There is a diffeomorphism  between  $\cQ(-1,-\tfrac12)$ and the cylinder $\Omega(-1)\times(-1,-\tfrac12)$. It follows from standard Bochner space properties (\cite[Section 25]{Wloka1987}  or \cite[Section 5.9.2]{Evans}) that the trace (in time) operator $ W(-1,-\tfrac12) \to L^2(\Omega(-1))$ is continuous and  combining this with \eqref{aux517} it follows that 
\begin{equation}\label{aux524}
\|u(-1)\|_{\Omega(-1)}\le C\, \|u\|_W, \quad u \in W,
\end{equation}
holds with a constant $C$ independent of $u$. 
For any smooth $u\in W$ and $t\in(-1,0)$, we obtain using \eqref{aux517}, \eqref{aux524}, the Reynolds transport theorem and $u_{|\partial \Omega(t)}=0$:
\begin{equation}\label{aux529} \begin{split}
\|u(t)\|_{\Omega(t)}^2 & = \|u(-1)\|_{\Omega(-1)}^2 + \int_{-1}^t\frac{d}{dt} \int_{\Omega(t)} u^2\,dx
 \\ & =\|u(-1)\|_{\Omega(-1)}^2 + 2\int_{\cQ(-1,t)} u_tu\,dq\le 
C\, \|u\|_W^2.  
\end{split} \end{equation}
%For the equality above we used the Reynolds transport theorem for the smoothly evolving domain and $u$ vanishing on the boundary. 
Since $u$ is smooth and  the constant $C$  is independent of $t$ and $u$, the estimate \eqref{aux529} also holds for $t=0$. For $t\in (0,T]$, the estimate $\|u(t)\|_{\Omega(t)}\le C\,\|u\|_W$ follows by the similar argument. The claim of the lemma now follows from the density of smooth functions in $W$.
\end{proof}

For a smooth function $f$ with  $f_{|\GQ}=0$ the Stokes theorem gives
\[
(f_t,1)_Q= (\nabla_{(x,t)}\cdot (f\mathbf{e}_t),1)_Q =\int_{\partial Q}
f\mathbf{e}_t\cdot \mathbf{n}_{\partial Q}\, dq =\int_{\Omega(T)} f(s,T)\, ds - \int_{\Omega(t_0)} f(s,t_0)\, ds.
\]
Applying the above identity to $f=uv$ with smooth $u$ and $v$ and  using the density result from Lemma~\ref{L3} we obtain the  integration by parts  identity
\begin{equation} \label{partint}
		\la u_t,v\ra +\la v_t,u\ra   = \int_{\Omega(T)} u(x,T) v(x,T)\, dx - \int_{\Omega(t_0)} u(x,t_0) v(x,t_0)\, dx\quad \text{for all}~~u,v \in W,
\end{equation}
where $u$ and $v$ terms on the right hand side are understood in terms of traces, cf. Lemma~\ref{L4}.
In the case of vanishing or emerging island, the formula \eqref{partint} appears without $\Omega(T)$ or $\Omega(t_0)$ terms, respectively.

\section{Well-posedness} \label{s:wellP} 
Using the properties of the spaces $H_0$ and $W$ derived in Section~\ref{s:spaceW} we now study well-posedness of a class of parabolic problems on the space-time domain $\cQ$. We remind assumptions that were needed up to now and collect them here as
\begin{assumption} \rm The  space-time domains $\cQ$ is defined by the subzero levels of a smooth level set function $\phi$, cf. \eqref{subzero}--\eqref{eq:smooth}. We assume that there is an isolated  nondegenerate critical point $(x_c,t_c)$, cf. \eqref{critpoint}--\eqref{critpoint2}, for which $\frac{\partial \phi}{\partial t}(x_c,t_c) \neq 0$ holds.  We consider all possible cases in   classification \ref{Class}, except for the cases 2c) and 3d).
\end{assumption}
\medskip

The integral identity \eqref{integral} suggests the following weak formulation of the heat equation \eqref{transport}--\eqref{Dirichlet}:
For given $f \in H^{-1}_0$ and $u_0 \in L^2(\Omega(t_0))$, find $u \in W$ such that $u|_{t=t_0}=u_0$ and
\begin{equation}\label{weak1}
	\langle u_t, v\rangle + (\nabla u, \nabla v)_\cQ
	= \langle f, v\rangle
	\qquad \forall\, v \in H_0 .
\end{equation}

To show well-posedness of \eqref{weak1}, we first transform it into a problem with homogeneous initial condition. This step is not needed in the emerging-island case. To this end, consider the decomposition $u=\widetilde{u}+u^0$, with some $u^0 \in W$ such that $u^0|_{t=t_0}=u_0$.

For example, one may set $u^0=\eta\, \hat u^0\circ(\Phi_{-})^{-1}$ for \rev{$t\le t_\ast$ and $u^0=0$ for $t> t_\ast$, with $t_\ast=(t_0+t_c)/2$}. Here $\Phi_{-}$ is the diffeomorphism from \eqref{Diff} and $\hat u^0$ is the solution of the heat equation in the cylindrical domain $(t_0,T)\times\Omega(t_0)$ with homogeneous Dirichlet boundary condition and initial value $u_0$. Hence,   $\hat u^0 \in L^2([t_0,T];H^1_0(\Omega(t_0)))$ and $\hat u^0_t \in L^2([t_0,T];H^{-1}(\Omega(t_0)))$ holds.  The function $\eta$ is a cut off function such that $\eta\in C^\infty([\rev{t_0,T}])$ and satisfies $\operatorname{supp}(\eta)\subset[\rev{t_0,t_\ast})$ and $\eta(t_0)=1$. 
\rev{Since $\Phi_{-}$ is smooth on $[t_0,t_\ast]$ and multiplying  $\eta$ localizes the construction away from the critical time  $t_c$, it follows $u^0 \in W$. }

Consider the closed subspace
\[
W_0 := \{\, v \in W \mid v(\cdot,t_0)=0 \text{ on } \Omega(t_0) \,\}.
\]
For the emerging-island case we take $W_0=W$. The space $W_0$ is well-defined, since functions in $W$ possess well-defined traces on $\Omega(t)$ for any $t\in[t_0,T]$, see Lemma~\ref{L4}.

Then $\widetilde u \in W_0$ satisfies \eqref{weak1} with modified right-hand side $\widetilde f = f - (u^0)_t + \Delta u^0 \in H_0^{-1}$.
We therefore take $W_0$ as the solution space. We allow for the more general linear parabolic problem:

Given $f\in H^{-1}_0$, find $u \in W_0$ such that
\begin{equation}\label{weakformu}
	\langle u_t,v\rangle + a(u,v) = \langle f,v\rangle
	\qquad \forall\, v \in H_0,
\end{equation}
where $a(\cdot,\cdot)$ is a continuous bilinear form on $H_0\times H_0$, i.e, $|a(u,v)|\leq \Gamma_a \|u\|_H \|v\|_H$ for all $u,v \in H_0$, that also satisfies the G{\aa}rding-type condition:
for any $g\in C([t_0,T])$ with $g>0$ on $[t_0,T]$,
\begin{equation}\label{Garding}
	a(u,gu)\ge c_0\|u\|_H^2 - c_1 (u,gu)_\cQ
	\qquad \forall\, u\in H_0,
\end{equation}
with $c_0=c_0(g)>0$ independent of $u$, and $c_1\ge0$ independent of $u$ and $g$.

\begin{example}\rm
	For the heat equation, \eqref{Garding} holds with $c_1=0$ and $c_0=(C_P^2+1)^{-1} \min_{[t_0,T]} g$, where $C_P>0$ is the constant from the uniform Poincaré inequality \eqref{UPoincare}.
	
Another example is the advection–diffusion equation
\begin{equation} \label{transport1}
	\frac{\partial u}{\partial t} + \operatorname{div}(u\bw) - \Delta u = f
	\quad \text{on } \Omega(t), \; t\in(t_0,T],
\end{equation}
supplemented with \eqref{Dirichlet} and an initial condition.
The G{\aa}rding condition holds if $\bw$ is such that
\begin{equation}\label{essinf}	
	\mbox{essinf}_\cQ\operatorname{div}\bw>-\infty
\end{equation}
	 holds.
	To verify $H$-continuity of $a(\cdot,\cdot)$, let $d(x,t)=\operatorname{dist}(x,\partial\Omega(t))$ and assume that the classical Hardy inequality on $H_0^1(\Omega(t))$ holds uniformly in time. Then
	\[
	|(\operatorname{div}(u\bw),v)_\cQ| =	|(u\bw,\nabla v)_\cQ|
	\le
	\|v\|_H\|u\bw\|_\cQ
	\le C\|v\|_H\|u\|_H\|d\bw\|_{L^\infty(\cQ)}.
	\]
	Hence $a(\cdot,\cdot)$ is continuous if $|d\bw|$ is essentially bounded in $\cQ$. We note that for the domain velocity $\bw=V$ defined in \eqref{trajectories} the latter condition is satisfied, while \eqref{essinf} holds for the cases 2a), 3a) if $\phi_t(0,0)<0$ and 3d) with $\phi_t(0,0)>0$. 
\end{example}
\smallskip

In the remainder of this section we show that the variational problem \eqref{weakformu} is well-posed.
Our analysis is based on the continuity and inf-sup conditions, cf.~\cite{Ern04}. %The arguments are standard, but we include them for completeness.

The continuity property is straightforward:
	\begin{equation}\label{contin}
| \langle u_t,v\rangle + a(u,v)| \le C_a \|u\|_W \|v\|_H \quad \text{for all}~~u \in W,~v \in H_0, \quad C_a:=\sqrt{1+\Gamma_a^2}.
	\end{equation}

The next two lemmas establish the necessary inf-sup conditions, with proofs following the arguments of~\cite[Section~4]{olshanskii2014eulerian}.

\begin{lemma}\label{la:infsup}
	The inf-sup inequality
	\begin{equation}\label{infsup}
		\inf_{0\neq u \in  W_0 }~\sup_{ 0\neq v \in {H_0}} \frac{\langle u_t,v\rangle + a(u,v)}{\|u\|_W\|v\|_H} \geq c_s
	\end{equation}
	holds with some $c_s>0$.
\end{lemma}
\begin{proof} 
	Take $u\in  W_0 $ and let  $u_\gamma=e^{\rev{\gamma (t_0-t)}}u\in H_0$, with some $\gamma\ge0$.	From \eqref{partint} and  $u(t_0)=0$, we infer
	\[
	\langle  u_t, u_\gamma \rangle= \tfrac12\left(\langle  (u_\gamma)_t, u \rangle+ \langle  u_t, u_\gamma \rangle+ \gamma(u,u_\gamma)_Q\right) \ge \tfrac\gamma2 (u,u_\gamma)_Q
	.\]
	From this and the   G{\aa}rding condition we get 
	\begin{equation}\label{eq:proofinfsup1}
		\la u_t, u_\gamma\ra +  a(u, u_\gamma) \ge  c_0\|u\|_H^2,\quad  c_0=c_0(\gamma)>0,
	\end{equation}
	if $\gamma=2c_1$.
	This establishes the control of $\|u\|_H$ on the right-hand side of the inf-sup inequality. We also need control of $\| u_t\|_{H^{-1}_0}$ to bound the full norm $\|u\|_W$.  
	
	By Riesz' representation theorem, there is a unique $z\in H_0$ such that $\la u_t,v\rangle = (z, v)_H$ for all $v\in H_0$, and $\|z\|_H = \| u_t\|_{H^{-1}_0}$ holds. Thus we obtain
	\[
	\la u_t,z\rangle = (z,z)_H = \| u_t\|_{H^{-1}_0}^2.
	\]
	Therefore,  we get
	\begin{equation}\label{eq:proofinfsup2}
			\langle  u_t,z\rangle +  a(u, z)  = \|z\|_H^2 +  a(u,z) \ge \|z\|_H^2 - \tfrac12 C_a^2 \|u\|_H^2 - \tfrac12\|z\|_H^2 
			 = \tfrac12\| u_t\|_{H^{-1}_0}^2 -  \tfrac12 C_a^2\|u\|_H^2.
	\end{equation}
	This establishes control of $\|u_t\|_{H^{-1}_0}$ at the expense of the $H$-norm, which is controlled in \eqref{eq:proofinfsup1}.
	Letting $v= z + \mu u_\gamma\in H_0$ with some sufficiently large parameter $\mu \geq 1$, we have the estimate
	\begin{equation}\label{aux20}
		\|v\|_H\le \|z\|_H + \mu\|u_\gamma\|_H \le \| u_t\|_{H^{-1}_0} + \mu \|u\|_H\le \mu \sqrt{2}\|u\|_W.
	\end{equation}
	Taking $\mu:= \frac{1+C_a^2}{2 c_0}$,  we conclude from  \eqref{eq:proofinfsup1}, \eqref{eq:proofinfsup2} and \eqref{aux20} that
	\[
	\la u_t,v\rangle +  a(u,v) \ge \tfrac12\|u_t\|_{H^{-1}_0}^2 +(\mu  c_0 - \tfrac12 C_a^2)\|u\|_H^2= \tfrac{1}{2}\|u\|_W^2 \geq \tfrac{1}{2 \sqrt{2}} \tfrac{1}{\mu} \|u\|_W\|v\|_H.
	\]
	This completes the proof.
\end{proof}
\smallskip

\begin{lemma} \label{la:leminj}
	If $\langle u_t,v\rangle + a(u,v)=0$  for some $v\in H_0$ and all $u \in  W_0 $, then $v=0$.
\end{lemma}
\begin{proof}
	Assume there is $v \in H_0$ such that $\la u_t,v\rangle =- a(u,v) $ for all $u \in  W_0 $. For all $ u \in C_c^1(\cQ) \subset  W_0 $ we have by definition
	\begin{align} \label{AA6}
		\langle v_t,u\rangle & = - \int_{\rev{t_0}}^T \int_{\Omega(t)} v  u_t  \, dx \, dt
		= -\langle u_t,v\rangle    =  a(u,v).
	\end{align}
	Since the functional  $u \to a(u,v)$ is in $H_0^{-1}$,  we conclude $ v_t \in H_0^{-1}$, and thus $v \in W$ holds. From \eqref{AA6} and the density result in Theorem~\ref{densityH0} it follows that
	\begin{equation}\label{aux2}
	\langle v_t,u\rangle=  a(u,v)\quad \text{ for all}~~u \in H_0.
	\end{equation}		
	Combining  this with $\la u_t,v\rangle =- a(u,v) $ for all $u \in  W_0 $  and  using \eqref{partint} we obtain
	\begin{equation}\label{aux1106}
	0=  \la v_t,u\rangle +\la u_t,v\ra= \int_{\Omega(T)} u(x,T) v(x,T)\, dx \quad \text{for all}~~u \in  W_0 .
	\end{equation}
	\rev{Since any $\xi\in C_c^\infty(\Omega(T))$ can be extended to an element of  $W_0$, e.g., by letting $u= \eta\, \xi\circ(\Phi_{+})^{-1}$ with a suitable smooth cut-off function $\eta$, traces of elements from  $W_0$ are dense in $L^2(\Omega(T))$.}
	This and \eqref{aux1106} imply that $v(\cdot,T)=0$ on $\Omega(T)$ (\rev{for the vanishing island we have $\Omega(T)=\emptyset$}). We proceed as in the first step of the proof of Lemma \ref{la:infsup}. We take in \eqref{aux2} $u=v_\gamma= e^{-\gamma t}v$, with $\gamma=-2 c_1$ (cf. \eqref{Garding}), and use \eqref{partint} 
	and the G{\aa}rding condition to obtain
	\begin{align*}
		0  &=\la v_t,v_\gamma \ra- a(v_\gamma ,v)  = \tfrac12( \langle  v_t, v_\gamma  \rangle + \langle  (v_\gamma)_t, v\rangle + \gamma(v,v_\gamma)_Q)  - a(v_\gamma ,v)   \\
		&\leq    \tfrac\gamma2(v,v_\gamma)_Q-  a(v_\gamma ,v)\le    - c_0 \|\rev{v_\gamma}\|^2_H, \quad\text{with}~c_0>0.
	\end{align*}
	We conclude \rev{$v_\gamma=0$ and so} $v=0$.
\end{proof}
\medskip

As a direct consequence of the preceding two lemmas and \eqref{contin} we obtain the following well-posedness result.

\begin{theorem} \label{mainthm1}
	For any $f\in H^{-1}_0$, the \rev{abstract} problem \eqref{weakformu} has a unique solution $u\in  W_0 $. This solution satisfies the a-priori estimate
	\[
	\|u\|_W \le c_s^{-1} \|f\|_{H^{-1}_0}.
	\]
\end{theorem}

\rev{For the model example of the heat equation \eqref{transport} this theorem yields a well-posedness result, which we formulate in the following corollary.
\begin{corollary} \label{corr1}
	For any $f\in H^{-1}_0$ and $u_0\in L^2(\Omega(t_0))$ (no initial condition is prescribed for the case of island creation), the problem \eqref{weak1} has a unique solution $u\in  W $ satisfying 
	\[
	\|u\|_W \le C (\|f\|_{H^{-1}_0}+\|u_0\|_{L^2(\Omega(t_0))}).
	\]
\end{corollary}	
}

We note that, since compactly supported functions belong to $H_0$, every weak solution is also a distributional solution of the parabolic problem \eqref{transport}. Conversely, by the density result, any smooth or merely distributional solution of \eqref{transport}  satisfies \eqref{weak1}.

%==============================================================================
\bibliographystyle{siam}
\bibliography{literatur}{}
% ==============================================================================

\appendix

\section{Proof of Theorem~\ref{densityH}} \label{AppA} It turns out that the arguments used in \cite{Adams2003} for proving density of $C^\infty(\overline \Omega)$ in the standard Sobolev space $W^{m,p}(\Omega)$ also apply to the anisotropic Sobolev space $H$, where only weak partial derivatives in certain coordinate directions are required to exist. We explain this in more detail. The numbering we use below, for example labeling of Lemmas or formula,  refers to \cite{Adams2003}. Instead of the notation $D^\alpha$ used for a general (maybe higher order) partial derivative used in \cite{Adams2003} we use the special case $\partial_i$ to denote a first order partial derivative with respect to $x_i$. Density of $C^\infty(\overline \Omega)$ in $W^{m,p}(\Omega)$ is proved in Theorem 3.22. It is assumed that the domain $\Omega$ satisfies the segment condition. The space-time domain $\cQ \in \R^{d+1}$ satisfies this condition. We explain the main ingredients used in the proof of this theorem and how these also apply to the anisotropic space $H$. It is convenient to use the notation $H=H(Q)$. \\
A standard mollifier $J_\epsilon$ is defined in 2.28. Let $\cO$ be a collection of open sets in $\R^{d+1}$ that cover $\cQ$. A $C^\infty$-partition of unity $\Psi$ for $\cQ$ subordinate to $\cO$ is defined in Theorem 3.15. 
\\
{\sc Lemma 3.16 (modified).} If $\cQ'$ is a subdomain with compact closure in $\cQ$, then $\lim\limits_{\epsilon \to 0} J_\epsilon \ast u= u$ in $H(\cQ')$.
\begin{proof} The arguments used in \cite{Adams2003} are applicable without modifications.
\end{proof}
\\
{\sc Theorem 3.22 (modified)}. The space $C^\infty(\overline \cQ)$ is dense in $H(\cQ)$.
\begin{proof}
We follow the arguments used in \cite{Adams2003}. The first part of the proof in \cite{Adams2003} is superfluous because the domain $\cQ$ is bounded. Let $u \in H(\cQ)$ be given. We extend $u$ by zero outside $\cQ$. Define $K:=\{\, y \in \cQ~|~u(y)\neq 0\,\}$, $F:=\overline{K}\setminus \ \left( \cup_{y \in \partial \cQ} U_y\right)$, where $\{U_y\}$ is the collection of open sets introduced in the definition of the segment condition.
There exists an open subset $U_0$ such that $F \subset \subset U_0 \subset \subset \cQ$. Furthermore $\overline{K} \subset U_0 \cup U_1 \cup \ldots \cup U_k$, where $(U_j)_{1 \leq j \leq k}$ is a subset of the collection $\{U_y\}$. Moreover, there are other open sets $V_0,V_1, \ldots, V_k$ such that $V_j \subset \subset U_j$ for $0 \leq j \leq k$ but still $\overline{K} \subset V_0 \cup V_1 \cup \ldots \cup V_k$. Let $\Psi$ be a $C^\infty$-partition of unity  subordinate to  $\{\, V_j~|~0\leq j \leq k\,\}$, and let $\psi_j$ be the sum of the finitely many functions $\psi \in \Psi$ whose supports lie in $V_j$. Define $u_j:=\psi_ju$. Suppose that for each $j$ we can find $\phi_j \in C_0^\infty(\R^{d+1})$ such that
\begin{equation} \label{eq7}
  \|u_j -\phi_j\|_H= \left(\|u_j-\phi_j\|_\cQ^2 + \sum_{i=1}^d\|\partial_i u_j -\partial_i \phi_j\|_{\cQ}^2\right)^\frac12 < \epsilon/(k+1).
\end{equation}
Then with $\phi:=\sum_{j=0}^k \phi_j \in C_0^\infty(\R^{d+1})$ we obtain
\[
 \|u- \phi\|_H \leq \sum_{j=0}^k\|u_j-\phi_j\|_H < \epsilon,
\]
which proves the claim. It remains to verify \eqref{eq7}. A function $\phi_0 \in C_0^\infty(\R^{d+1})$ satisfying \eqref{eq7} for $j=0$ can be found using Lemma~3.16 (modified) since ${\rm supp}(u_0) \subset V_0 \subset \subset \cQ$. Now consider $1 \leq j \leq k$. Note that $u_j \in H(\R^{d+1} \setminus \Gamma)$ with $\Gamma=\overline{V}_j\cap \partial \cQ$. Let $z\in \R^{d+1}$ be the vector associated with $U_j$ in the definition of the segment condition, and define
 for $\tau \geq 0$, $\Gamma_\tau:=\{\, y-\tau z~|~y \in \Gamma\,\}$. For $\tau$ sufficiently small we have $\Gamma_\tau \subset U_j$  and $\Gamma_\tau \cap \overline{\cQ} = \emptyset$. For $u_{j,\tau}(y):=u_j(y+\tau z)$ we then have $u_{j,\tau} \in H(\R^{d+1} \setminus \Gamma_\tau)$. Translation is continuous in $L^2(\cQ)$ so $\partial_i u_{j,\tau} \to \partial_i u_j$ in $L^2(\cQ)$ as $\tau \to 0$ for $1 \leq i \leq d$. Thus $u_{j,\tau} \to u_j$ in $H(\cQ)$ as $\tau \to 0$. Hence, it is sufficient to find $\phi_j \in C_0^\infty(\R^{d+1})$ such that $\|u_{j,\tau}-\phi_j\|_H$ is sufficiently small. Note that $(\cQ \cap U_j) \subset \subset \R^{d+1} \setminus \Gamma_\tau$, and thus by Lemma 3.16 (modified) we can take $\phi_j = J_\delta \ast u_{j,\tau}$ for suitably small $\delta > 0$.  
\end{proof}

\section{Proof of Lemma~\ref{lemA}} \label{AppB} We  consider the case of domain splitting in three dimensions ($d=3$): $\phi(x,t)=-x_1^2 +x_2^2 +x_3^2+t$. The case of domain merging is then also covered by replacing $t$ by $-t$. Note that  $-3 \eps \leq \phi(x,t)\leq 0$ iff $x_2^2+x_3^2 +t \leq x_1^2 \leq x_2^2+x_3^2 +t  + 3 \eps$. We use polar coordinates $(\psi,r)$ in the $(x_2,x_3)$-plane.
First consider $t \geq 0$. Then we have 
\[ \|(\nabla\theta_\eps)u\|_{\cQ_\eps^{c} \cap (t \geq 0)}^2  \leq c \eps^{-2}\int_0^\delta \int_0^{2\pi} \int_{0}^\delta \int_{\sqrt{r^2 +t}}^{\sqrt{r^2 +t+3 \eps}} u^2 (x_1^2+ r^2)\, dx_1 r \, dr  \, d\psi \, dt.
\]
The inner integral can be estimated  as in \eqref{HH1}, which results in the desired estimate \eqref{E6}. For the case $t\leq 0$ we note $-3 \eps \leq \phi(x,t)\leq 0$ iff $x_1^2-t -3 \eps \leq x_2^2 +x_3^2 \leq x_1^2-t$. We  obtain,  similar to \eqref{HH2},
\begin{equation} \label{ApHH2}
 \|(\nabla\theta_\eps)u\|_{\cQ_\eps^{c} \cap (t \leq 0)}^2  \leq c \eps^{-2}\int_{-\delta}^0 \int_{-\delta}^\delta \int_0^{2 \pi}\int_{[x_1^2 -t-3\eps]_+^\frac12}^{\sqrt{x_1^2 -t}} u^2 (x_1^2+ r^2) r \, dr \, d \psi \, dx_1 \, dt.
\end{equation}
For the inner integral we obtain, by using the same estimates as in \eqref{HH3}-\eqref{HH4} and the Hardy inequality \eqref{ResP1}:
\begin{align*}
 \eps^{-2}\int_{[x_1^2 -t-3\eps]_+^\frac12}^{\sqrt{x_1^2 -t}} u^2 (x_1^2+ r^2) r \, dr  & \leq c \int_{[x_1^2 -t-3\eps]_+^\frac12}^{\sqrt{x_1^2 -t}} \left(\frac{u}{\sqrt{x_1^2 -t} - r}\right)^2 r  \, dr   \\ & \leq c \int_{[x_1^2 -t-3\eps]_+^\frac12}^{\sqrt{x_1^2 -t}} \left(\frac{\partial u}{\partial r}\right)^2 r  \, dr,
\end{align*}
and using this in \eqref{ApHH2} yields the desired estimate \eqref{E6}. \\
We consider the case 3c) of the creation of a hole through the domain (the case of a vanishing hole follows from this by replacing $t$ by $-t$): $\phi(x,t)=-x_1^2 -x_2^2 +x_3^2+t$. If we use polar coordinates $(\psi,r)$ in the $(x_1,x_2)$-plane then  $-3 \eps \leq \phi(x,t)\leq 0$ iff $r^2-t -3 \eps \leq x_3^2 \leq r^2-t$. First consider $t \geq 0$. Then we have 
\begin{equation} \label{P0} \|(\nabla\theta_\eps)u\|_{\cQ_\eps^{c} \cap (t \geq 0)}^2  \leq c \eps^{-2}\int_0^\delta \int_0^{2\pi} \int_{\sqrt{t}}^\delta \int_{[r^2-t -3\eps]_+^\frac12}^{\sqrt{r^2 -t}} u^2 (x_3^2+ r^2)\, dx_3 r \, dr  \, d\psi \, dt.
\end{equation}
We use the same estimates as in \eqref{HH4}, which yields $(\sqrt{r^2-t}-x_3)^2(x_3^2+ r^2) \leq 18 \eps^2$ for 
$x_3$ with $[r^2-t -3\eps]_+^\frac12 \leq x_3 \leq \sqrt{r^2 -t}$. Using this in \eqref{P0} and applying the Hardy inequality \eqref{ResP1} yields the estimate \eqref{E6}.  The case $t \leq 0$ can be treated very similarly. 
\end{document}